\documentclass[11pt]{amsart}
\usepackage[utf8]{inputenc}
\usepackage[english]{babel}
\usepackage{graphicx}
\usepackage{amssymb}
\usepackage{amsmath}
\usepackage[mathscr]{euscript}
\usepackage{accents}

\usepackage{tikz}
\usetikzlibrary{decorations.markings}

\usepackage{microtype} 

\theoremstyle{plain}	
\newtheorem{theorem}{Theorem}[section]
\newtheorem{proposition}[theorem]{Proposition}

\newtheorem{lemma}[theorem]{Lemma}
\theoremstyle{definition} 
\newtheorem{definition}[theorem]{Definition}
\newtheorem{example}[theorem]{Example}

\title[DW Invariants of Seifert Surfaces in 3-Manifolds]{Dijkgraaf-Witten Type Invariants of Seifert Surfaces in 3-Manifolds}
\author{I.\ J.\ Lee and D.\ N.\ Yetter}

\begin{document}
\tikzset{midarrow/.style={
		decoration={markings,
			mark= at position 0.5 with {\arrow{#1}},
		},
		postaction={decorate}
	}
}

\begin{abstract}  We introduce defects, with internal gauge symmetries, on a knot and Seifert surface to a knot into the combinatorial construction of finite gauge-group Dijkgraaf-Witten theory.  The appropriate initial data for the construction are certain three object categories with coefficients satisfying a partially degenerate cocycle condition.

\end{abstract}

\maketitle

\section{Introduction}

Throughout we work exclusively in the PL category, all manifolds are compact, and we write composition in diagrammatic order. The goal of this paper is to construct state-sum invariants of triples $K \subset \Sigma \subset M$, in which $M$ is an oriented 3-manifold, $K$ an oriented knot or link in $M$ and $\Sigma$ an oriented surface with boundary embedded in $M$ with $\partial \Sigma = K$, extending and generalizing the state-sum construction of finite-gauge group Dijkgraaf-Witten theory given by Wakui \cite{W}.   To avoid confusion between link in the knot-theoretic sense and the link of a simplex in the sense of combinatorial topology, we will always refer to $K$ as ``the knot'', even though it might be a link in the knot-theoretic sense.   The arrangement $K \subset \Sigma \subset M$ is a very simple stratified space of the sort Crane and Yetter \cite{CY} called  ``starkly stratified ''.  However, as was the case in two dimensions (cf. \cite{DPY}), our constructions will require us to restrict our attention to flag-like triangulations.  The first task will, thus, be to show that the instances of extended bistellar moves which preserve flag-likeness suffice to give all PL homeomorphisms between stratified PL spaces of the form $K \subset \Sigma \subset M$.

From \cite{CY} recall

\begin{definition} A {\em starkly stratified space} is a PL space $X$ equipped with a filtration

\[ X_0 \subset X_1 \subset \ldots \subset X_{n-1} \subset X_n = X \]

\noindent satisfying

\begin{enumerate}
\item There is a triangulation $\mathcal T$ of $X$ in which each $X_k$ is a subcomplex.
\item For each $k = 1,\ldots n$ $X_k \setminus X_{k-1}$ is a(n open) $k$-manifold.
\item If $C$ is a connected component of $X_k \setminus X_{k-1}$, then $\mathcal T$ restricted
to $\bar{C}$ gives $\bar{C}$ the structure of a combinatorial manifold with boundary.
\item For each combinatorial ball $B_k$ with $\accentset{\circ}B_k \subset X_k \setminus X_{k-1}$, $\accentset{\circ}B_k$ admits a closed neighborhood $N$ given inductively as a cell complex as follows (although we require $B_k$ to be a combinatorial ball, the triangulation is then ignored):

\noindent $N = N_n$, where $N_m$ for $k \leq m \leq n$ is given inductively by 

\[ N_k = B_k\] 

\noindent and 

\[ N_{\ell+1} = N_\ell \cup \bigcup_{v \in S_\ell} L(v) \ast v \]

\noindent for $S_\ell$ a finite set of points in $X\setminus X_\ell$, andl $L(\cdot)$ a function on $S_\ell$
valued in 

\[ \{ L | L\; \mbox{\rm is a combinatorial ball and}\; B_k \subset L \subset N_\ell  \}\]

\noindent We will call such a neighborhood of the interior of a combinatorial ball lying in a stratum of the same dimension a {\em stark neighborhood}.

\end{enumerate}
\end{definition}

It is quite easy to show that for a knot or link $K$ with a Seifert surface $\Sigma$ in a 3-manifold $M$, the filtered space $\emptyset \subset K \subset \Sigma \subset M$ is a starkly stratified space -- the stark neighborhoods of open simplexes $\Sigma$ are simply the joins of their closure with two points of $M \setminus \Sigma$, one on each side of $\Sigma$, $\Sigma$ being, of course, two-sided, while those of open simplexes in $K$ are iterated joins, first with a point in $\Sigma$ and a point in $M \setminus \Sigma$ and then of with two points of $M \setminus \Sigma$, one on either side of $\Sigma$ (or more properly the union of $\Sigma$ with the first join, which is easily seen to also be an oriented surface with boundary). 

Throughout, despite allowing the possibility that is is actually a link with more than one component, we will refer to $K$ as ``the knot'' to avoid the ambiguity of the word ``link'' in the context of PL topology, as we will have cause to mention the link of a simplex in the PL sense fairly often in our exposition.

We will also need

\begin{definition}
A triangulation $\mathcal T$ of a stratified PL space 

\[ X_0 \subset X_1 \subset \ldots \subset X_{n-1} \subset X_n = X \]

\noindent is {\em flag-like} if each of the $X_i$ is a subcomplex, and moreover for each simplex $\sigma$ of
$\mathcal T$, the restriction of the filtration to the simplex, that is the distinct non-empty intersections in the sequence

\[ X_0\cap \sigma \subset X_1 \cap \sigma \subset \ldots \subset X_{n-1} \cap \sigma \subset \sigma \]

form a (possibly incomplete) simplicial flag.
\end{definition}

Observe that in any starkly stratified space, it follows from the iterative construction of a stark neighborhood, that the restriction of the filtration to any simplex of a stark neighborhood is a simplical flag.  Moreover, in general, flag-like triangulations are plentiful.  In particular, it is easy to show

\begin{lemma} If  $\mathcal T$ is a triangulation of a stratified PL space 

\[ X_0 \subset X_1 \subset \ldots \subset X_{n-1} \subset X_n = X \]

\noindent for which each $X_i$ is a subcomplex, then its barycentric subdivision $\beta X$ is flag-like.
\end{lemma}

As our construction will require a flag-like triangulation, the most direct way to show topological invariance will be show that combinatorial moves which preserve flag-likeness suffice to give all PL homeomorphisms between our stratified spaces $K \subset \Sigma \subset M$.

\begin{definition}
An Alexander move (stellar subdivision or stellar weld), a Pachner move (bistellar move), or an extended Pachner move (extended bistellar move in the sense of \cite{CY}) on a triangulation of a stratified space is {\em flag-like} if the restriction of the filtration to each simplex of both its initial and its final state is a simplicial flag.
\end{definition}

\section{From Alexander Moves to Extended Pachner Moves}

As observed in \cite{DPY}, any Alexander subdivision is flag-like, while there exist non-flag-like Alexander welds, but the theorem of \cite{DPY} that flag-like Alexander moves suffice to characterize PL homeorphism of stratified PL spaced equipped with flag-like triangulations hold in full generality:

\begin{theorem}
If $\mathcal T$ and ${\mathcal T}^\prime$ are flag-like triangulations of a stratified PL spaces

\[ X_0 \subset X_1 \subset \ldots \subset X_{n-1} \subset X_n = X  \]

\noindent and

\[ X_0^\prime \subset X_1^\prime \subset \ldots \subset X_{n-1}^\prime \subset X_n^\prime = X^\prime  , \]

\noindent respectively, then there is a stratification-preserving PL homeomorphism from $X$ to $X^\prime$ if and only if there is a sequence of flag-like Alexander moves starting with $\mathcal T$ and resulting in a triangulation of $X$ which is combinatorially equivalent to the triangulation ${\mathcal T}^\prime$ of $X^\prime$.
\end{theorem}

\begin{proof}
The proof is identical to that given in the special case in \cite{DPY}:  Sufficiency is clear, necessity follows from the full force of Alexander's theorem -- that two complexes are PL homoemorphic if and only if there is a common subdivision by Alexander subdivisions --  and the observation that subdivisions of flag-like triangulations are always flag-like:  starting at $\mathcal T$ and ${\mathcal T}^\prime$, apply Alexander subdivisions (necessarily flag-like) until a common combinatorial type of triangulation is reached.  The sequence in the theorem is then the sequence of subdivsions beginning at $\mathscr T$ followed by the welds reversing the subdivisions starting at
${\mathscr T}^\prime$ in reverse order.
\end{proof}

Working with Alexander moves in the context of state-sum models is difficult, since in dimensions three and greater there are infintely many combinatorial types of Alexander move, and what is more, the number of simplexes in the subdivided triangulation is quite large.  We therefore wish to reduce PL homeorphism of our triples $K \subset \Sigma \subset M$, equipped with flag-like triangulations to flag-like instances of the extended Pachner moves of \cite{CY}.

\begin{theorem} \label{sufficient_moves}
Every flag-like Alexander move relating triangulations of a knot, Seifert surface, 3-manifold triple 
$K \subset \Sigma \subset M$ can be accomplished by a sequence of flag-like extended Pachner moves.

Every flag-like Alexander move relating triangulation of a knot, Seifert surface, 3-manifold triple $K \subset \Sigma \subset M$ can be accomplished by a sequence of flag-like moves of the following types:
Pachner moves on tetrahedra, extended Pachner moves on triangles of the Seifert surface, and Alexander subdivision of an edge of the knot whose link has exactly three edges (a 3-6 move).
\end{theorem}

\begin{proof}
It suffices to show that flag-like Alexander subdivisions can be accomplished by sequences of flag-like extended Pachner moves. 

 For subdivisions of tetrahedra, subdivisions of triangles lying in $\Sigma$ and of
edges of $K$ with regular neighborhoods as depicted below,

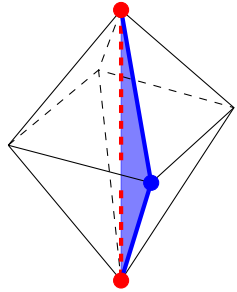
\begin{center}
\begin{tikzpicture}
\draw [fill=blue,blue, opacity=.5] (1.9,-0.5) --(1.5,-1.8) -- (1.5,1.8); 

\draw (0,0) --(1.9,-0.5) -- (3,0.5) -- (1.5,1.8)-- (0,0);
\draw (0,0) --(1.5,-1.8)-- (3,0.5);
\draw (1.9,-0.5)-- (1.5,-1.8);
\draw (1.5,1.8)-- (1.9,-0.5);
\draw [dashed] (0,0) -- (1.2,1.0) -- (3,0.5);
\draw [dashed] (1.5,1.8) -- (1.2,1.0) -- (1.5,-1.8);
\draw [dashed,red,ultra thick] (1.5,1.8) -- (1.5,-0.2);
\draw [dashed,red,ultra thick] (1.5,-0.6) -- (1.5,-1.8);

\draw [blue,ultra thick] (1.5,1.8)--(1.9,-0.5)--(1.5,-1.8); 
\draw [fill=red, red] (1.5,1.8) circle (1mm); 
\draw [fill=blue, blue] (1.9,-0.5) circle (1mm); 
\draw [fill=red, red] (1.5,-1.8) circle (1mm); 

\end{tikzpicture}
\end{center}

\noindent the first statement is immediate, since the Alexander move itself is an extended Pachner move.   Likewise the second statement is immediate for subdivisions of tetrahedra, triangles lying in $\Sigma$ and edges of $K$ with regular neighborhoods as depicted here:

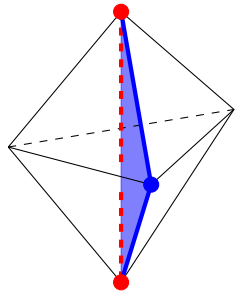
\begin{center}
\begin{tikzpicture}
\draw [fill=blue,blue, opacity=.5] (1.9,-0.5) --(1.5,-1.8) -- (1.5,1.8); 

\draw (0,0) --(1.9,-0.5) -- (3,0.5) -- (1.5,1.8)-- (0,0);
\draw [dashed] (0, 0) -- (3,0.5);
\draw (0,0) --(1.5,-1.8)-- (3,0.5);
\draw (1.9,-0.5)-- (1.5,-1.8);
\draw (1.5,1.8)-- (1.9,-0.5);
\draw [dashed,red,ultra thick] (1.5,1.8) -- (1.5,-0.2); 
\draw [dashed,red,ultra thick] (1.5,-0.6) -- (1.5,-1.8); 

\draw [blue,ultra thick] (1.5,1.8)--(1.9,-0.5)--(1.5,-1.8); 
\draw [fill=red, red] (1.5,1.8) circle(1mm); 
\draw [fill=blue, blue] (1.9,-0.5) circle(1mm); 
\draw [fill=red, red] (1.5,-1.8) circle(1mm); 
\end{tikzpicture}
\end{center}

For subdivisions of edges in $\Sigma$, the result follows by extending the Pachner moves
in the proof of the corresponding result in \cite{DPY} to a stark neighborhood formed by
taking the join with a point on either side of the surface.

Subdivisions of triangles with interior in $M \setminus \Sigma$ are easily seen to be accomplished
by a sequence of two flag-like Pachner moves, regardless of how the filtration restricts to the pair of tetrahedra sharing the triangular face:  first perform a 1-4 Pachner move in one of the tetrahedra, as a subdivision, of flag-like triangulation this move is necessarily flag-like.  The resulting complex
than has the face to be subdivided shared by the other of the original tetrahedra and one of those
from the subdivsion, whose interesection with lower-dimensional strata necessarily lies in the
boundary of triangle being subdivided.  The 2-3 Pachner move on this pair of tetrahedra is necessarily flag-like since the intersection with the lower dimensional strata lies entirely in the original tetrahedron, and 
performing it accomplished the desired subdivision.

The main difficulty in the proof of each statement comes from handling subdivisions of edges lying on the knot, the link of which has a different number of edges than that included in the set of moves, and subdivisions of edges with interior in the top dimensional stratum.

For the latter, we proceed as follows:  first, observe that the link of any given edge is a polygon, and the star of the edge consists of tetrahedra including an edge of the link and the given edge (as one of its pairs of disjoint edges) and triangles shared by two of the tetrahedra, having one vertex in the link and the given edge as their opposites side, and the faces of these.  As the edge has interior in the top dimensional stratum, and the triangulation is flag-like, the given edge has at most one vertex in a lower dimensional stratum (the knot or the Seifert surface), and the higher

Now, perform a 1-4  Pachner move on each tetrahedron in the star, as Alexander subdivisions of a flag-like triangulation, these moves are necessarily flag-like.  Having done this, consider the triangles which were shared by the original tetrahedra of the given edge.  They are now each shared by two tetrahedra which intersect the lower-dimensional strata only in the intersection of the triangle with the lower-dimensional -- by flag-likeness this must be empty, a single vertex (in $\Sigma$ or in $K$) or an edge with its end-points (all in $\Sigma$, all in $K$, or in $\Sigma$, except for a vertex in $K$) -- and thus the entire star of each of the triangles in the new triangulation intersects the lower-dimensional strata only in the empty set or the same vertex or edge.  It is immediate then that the 2-3 Pachner move on the star of each triangle is flag-like.

Now, perform all of the 2-3 Pachner moves on the stars of the triangles shared by the original tetrahedra.
The link of the given edge in this new triangulation has as its vertices exactly the new vertices introduced by the 1-4 moves of the first step, and as its edges precisely the new edges introduced by the 2-3 moves.  It is easy to see that the intersection of the star of the given edge in this new triangulation with the lower-dimensional strata is exactly the intersection of (the closure of) the given edge, that is, either empty, or a single vertex in either $\Sigma$ or $K$.  In any event, any Pachner moves performed within the star of the given edge in the last triangulation are necessarily flag-like.  By Cassali's \cite{C} improvement of Pachner's result, we can reach any triangulation of the star (as a manifold with boundary) with the same triangulation of the boundary by a squence of Pachner moves on the interior of the star, in particular we can perform the Alexander subdivision of the edge by a sequence of such moves.

Having subdivided the given edge in the modified triangulation, we must now apply more flag-like Pachner moves to obtain the Alexander subdivision of the given edge {\em in the original triangulation}.  Now, consider the cells in the most recently constructed triangulation within which we had performed the 2-3 Pachner moves.  These cells are still cells of the latest triangulation, and still intersect the lower-dimensional strata in such a way that any moves performed in their interiors are necessarily flag-like, and are still each the star of and edge introduced by one of the 2-3 moves, but this star now consists of four closed tetrahedra all sharing the edge -- being in particular the join of a pair of triangles sharing the edge with two vertices.  By Cassali \cite{C} we can perform the join with the two vertices of a 2-dimensional 2-2 Pachner move removing the edge introduced by the 2-3 move.

Now, observe that the last constructed triangulation is a proper subdivision the triangulation we had after applying the 1-4 moves (and thus of the original triangulation), and, moreover, in it, the triangle shared by the tetrahedra of the original star have been divided into two triangles, exactly as in the result of the desired Alexander move.   Applying Cassali's main theorem \cite{C} to the subdivision of each of the original tetrahedra (any moves within which much be flag-like) then obtains the desired triangulation by a sequence of flag-like Pachner moves.

Finally, we must reduce all Alexander subdivisions of an edge of the knot to flag-like extended Pachner moves on simplexes not on the knot, together with the subdivion of an edge of the knot with a four (resp. three) edge link to establish the first (resp. second) statement.  

For the second statement, consider first the case of reducing the Alexander subdivision of the edge in the first picture above to flag-like extended Pachner moves on simplexes not on the knot, together with the subdivion of an edge of the knot with a four three edge link.  The star of the edge to be subdivided is show below.

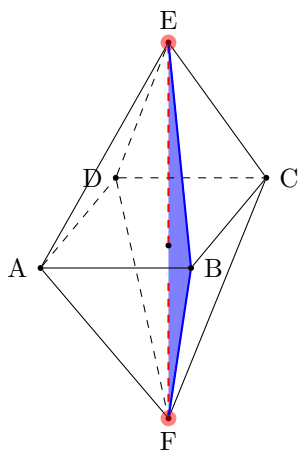
\begin{center}
\begin{tikzpicture}
		\fill[blue,blue, opacity=.5] (2,0)--(1.7,3)--(1.7,-2);
\node[circle, label=left:{\small A}, fill, inner sep=.8pt, outer sep=0pt] (A) at (0,0){};
	\node[circle, label=right:{\small B}, fill, inner sep=.8pt, outer sep=0pt] (B) at (2,0){};
	\node[circle, label=right:{\small C}, fill, inner sep=.8pt, outer sep=0pt] (C) at (3,1.2){};
	\node[circle, label=left:{\small D}, fill, inner sep=.8pt, outer sep=0pt] (D) at (1,1.2){};
	\node[circle, label=above:{\small E}, fill, inner sep=.8pt, outer sep=0pt] (E) at (1.7,3){};
	\node[circle, label=below:{\small F}, fill, inner sep=.8pt, outer sep=0pt] (F) at (1.7,-2){};
	\node[circle, fill, inner sep=.8pt, outer sep=0pt] (G) at (1.7,0.3){};
	\begin{scope}
		\draw (A)--(B)--(C);
		\draw [dashed] (A)--(D)--(C);
		\draw (A)--(E)--(B);
		\draw (E)--(C);
		\draw [dashed] (D)--(E);
		\draw (A)--(F)--(B);
		\draw (F)--(C);
		\draw [dashed] (D)--(F);
		\draw [dashed, thick,red] (E)--(G);
		\draw [dashed, thick,red] (1.7,-0.2)--(F);
		\fill[red, opacity=.5] (E) circle (3pt);
		\fill[red, opacity=.5] (F) circle (3pt);
		\draw[thick,blue,blue] (E)--(B)--(F);
	\end{scope}
\end{tikzpicture}
\end{center}

First, apply a 2-3 Pachner move to two tetrahedra $AFED$ and $DFCE$, sharing the triangle $DEF$.  Observe that this move is flag-like.  Giving the triangulation given within the original star below.

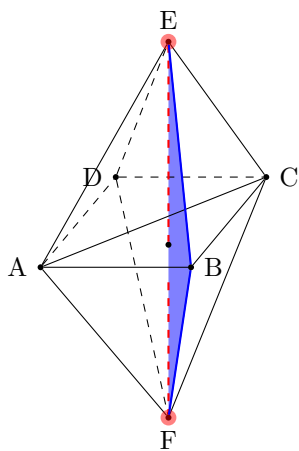
\begin{center}
\begin{tikzpicture}
		\fill[blue,blue, opacity=.5] (2,0)--(1.7,3)--(1.7,-2);
\node[circle, label=left:{\small A}, fill, inner sep=.8pt, outer sep=0pt] (A) at (0,0){};
	\node[circle, label=right:{\small B}, fill, inner sep=.8pt, outer sep=0pt] (B) at (2,0){};
	\node[circle, label=right:{\small C}, fill, inner sep=.8pt, outer sep=0pt] (C) at (3,1.2){};
	\node[circle, label=left:{\small D}, fill, inner sep=.8pt, outer sep=0pt] (D) at (1,1.2){};
	\node[circle, label=above:{\small E}, fill, inner sep=.8pt, outer sep=0pt] (E) at (1.7,3){};
	\node[circle, label=below:{\small F}, fill, inner sep=.8pt, outer sep=0pt] (F) at (1.7,-2){};
	\node[circle, fill, inner sep=.8pt, outer sep=0pt] (G) at (1.7,0.3){};
	\begin{scope}
		\draw (A)--(B)--(C);
		\draw [dashed] (A)--(D)--(C);
		\draw (A)--(C);
		\draw (A)--(E)--(B);
		\draw (E)--(C);
		\draw [dashed] (D)--(E);
		\draw (A)--(F)--(B);
		\draw (F)--(C);
		\draw [dashed] (D)--(F);
		\draw [dashed, thick,red] (E)--(G);
		\draw [dashed, thick,red] (1.7,-0.2)--(F);
		\fill[red, opacity=.5] (E) circle (3pt);
		\fill[red, opacity=.5] (F) circle (3pt);
		\draw[thick,blue,blue] (E)--(B)--(F);
	\end{scope}
\end{tikzpicture}
\end{center}

Then we can apply 3-6 move to the new star of the edge EF.

\begin{center}
\begin{tikzpicture}
		\fill[blue,blue, opacity=.5] (2,0)--(1.7,3)--(1.7,-2);
\node[circle, label=left:{\small A}, fill, inner sep=.8pt, outer sep=0pt] (A) at (0,0){};
	\node[circle, label=right:{\small B}, fill, inner sep=.8pt, outer sep=0pt] (B) at (2,0){};
	\node[circle, label=right:{\small C}, fill, inner sep=.8pt, outer sep=0pt] (C) at (3,1.2){};
	\node[circle, label=left:{\small D}, fill, inner sep=.8pt, outer sep=0pt] (D) at (1,1.2){};
	\node[circle, label=above:{\small E}, fill, inner sep=.8pt, outer sep=0pt] (E) at (1.7,3){};
	\node[circle, label=below:{\small F}, fill, inner sep=.8pt, outer sep=0pt] (F) at (1.7,-2){};
	\node[circle, label=below:{\tiny G}, fill, inner sep=.8pt, outer sep=0pt] (G) at (1.7,0.3){};
	\begin{scope}
		\draw (A)--(B)--(C);
		\draw [dashed] (A)--(D)--(C);
		\draw (A)--(C);
		\draw (A)--(E)--(B);
		\draw (E)--(C);
		\draw [dashed] (D)--(E);
		\draw (A)--(F)--(B);
		\draw (F)--(C);
		\draw [dashed] (D)--(F);
		\draw [dashed, thick,red] (E)--(G);
		\draw [dashed, thick,red] (1.7,-0.2)--(F);
		\draw [thick] (A)--(G);
		\draw [thick] (B)--(G);
		\draw [thick] (C)--(G);
		\fill[red, opacity=.5] (E) circle (3pt);
		\fill[red, opacity=.5] (F) circle (3pt);
		\draw[thick,blue,blue] (E)--(B)--(F);
	\end{scope}
\end{tikzpicture}
\end{center}

Subdividing the edge AC using flag-like Pachner moves as in the argument above, then reversing a like sequence of flag-like Pachner moves to introduce an edge DG then gives the desired final configuration shown below.

\begin{center}
\begin{tikzpicture}
		\fill[blue,blue, opacity=.5] (2,0)--(1.7,3)--(1.7,-2);
\node[circle, label=left:{\small A}, fill, inner sep=.8pt, outer sep=0pt] (A) at (0,0){};
	\node[circle, label=right:{\small B}, fill, inner sep=.8pt, outer sep=0pt] (B) at (2,0){};
	\node[circle, label=right:{\small C}, fill, inner sep=.8pt, outer sep=0pt] (C) at (3,1.2){};
	\node[circle, label=left:{\small D}, fill, inner sep=.8pt, outer sep=0pt] (D) at (1,1.2){};
	\node[circle, label=above:{\small E}, fill, inner sep=.8pt, outer sep=0pt] (E) at (1.7,3){};
	\node[circle, label=below:{\small F}, fill, inner sep=.8pt, outer sep=0pt] (F) at (1.7,-2){};
	\node[circle, label=below:{\tiny G}, fill, inner sep=.8pt, outer sep=0pt] (G) at (1.7,0.53){};
	\begin{scope}
		\draw (A)--(B)--(C);
		\draw [dashed] (A)--(D)--(C);
		\draw (A)--(E)--(B);
		\draw (E)--(C);
		\draw [dashed] (D)--(E);
		\draw (A)--(F)--(B);
		\draw (F)--(C);
		\draw [dashed] (D)--(F);
		\draw [dashed, thick,red] (E)--(G);
		\draw [dashed, thick,red] (1.7,-0.2)--(F);
		\draw [thick] (A)--(G);
		\draw [thick] (B)--(G);
		\draw [thick] (C)--(G);
		\draw [thick] (D)--(G);
		\fill[red, opacity=.5] (E) circle (3pt);
		\fill[red, opacity=.5] (F) circle (3pt);
		\draw[thick,blue,blue] (E)--(B)--(F);
	\end{scope}
\end{tikzpicture}
\end{center}

Now for edges of the knot with links of greater than four edges, the number of edges in the link can be reduced by a 2-3 move, as was done in the initial step of the construction just completed, and after the application of the 3-6 move, a sequence of pairs of a subdivision of an edge introduced by a 2-3 move, followed by a reverse subdivision to introduce a radial edge from the subdivision point introduced by the 3-6 move, will accomplish the desired Alexander subdivision.   Thus we establish the second statement.  The same argument will also reduce Alexander subdivisions of edges on the knot with links of more than four edges to sequences of flag-like extended Pachner moves.  

To establish the first statement, it thus remains only to show that the 3-6 move can be accomplished by a sequence of flag-like extended Pachner moves.  Consider an edge of the knot with a three edge link.  The link consists of two edges incident with the Seifert surface and one not incident with the Seifert surface.  As shown above, the edge not incident with the Seifert surface can be subdivided by a sequence of flag-like Pachner moves.   Having done this, the edge of the knot now has a 4 edge link, and can be subdivided by the extended Pachner move (4-8 move).  It remains then to weld two edges which had been the edge of the link not incident with the Seifert surface.  This, however, is the reverse of a subdivision of an edge not incident with either the knot or the Seifert surface, so as shown above can be done by a sequence of flag-like Pachner moves, thus establishing the result. 
\end{proof}

\section{Initial Data:  Untwisted Case}

The initial data describing local states in Wakui's construction \cite{W} of  state-sum for finite-gauge group Dikjkgraaf-Witten theory is a finite group, $G$, together with a 3-cocycle $\alpha:G^3 \rightarrow K^\times$ valued in the multiplicative group of a field $K$.

In \cite{DPY}, 2-dimensional Dijgraaf-Witten theory was extended to pairs $C\subset \Sigma$ of a closed curve lying in a surface, using local states given by a triple $(H,X,G)$ of two groups $H$ and $G$ and a set $X$ with commuting actions of the groups on the left and right, respectively, together with $K^\times$-valued coefficients extending a 2-cocycle on $G$ and satisfying 2-cocycle-like conditions.  

For now we defer consideration of the analogue of a cocycle and describe the initial data sufficient for the ``untwisted'' construction, which will turn out to be equivalent to the construction involving an analogue of a cocycle which is identically $1$.

Brief consideration of the conditions on the triple $(H,X,G)$ show that the data can be regarded as a two-object category in which the elements of $H$ are the endomorphisms of one object, which we denote 1, the elements of $G$ the endomorphisms of the other, which we denote 2, there are no arrows from 2 to 1 and the elements of $X$ are the arrow from 1 to 2 -- the commutativity of the actions being part of the associativity of composition in the category.

\begin{definition} A $\mathcal P$-{\em  parcel} is a small category $\mathcal C$ equipped with a surjective conservative functor $U:{\mathcal C}\rightarrow {\mathcal P}$, which is injective on objects, to a poset $\mathcal P$.
\end{definition}

A finite $\mathsf 3$-parcel, where $\mathsf 3$ is the three-element chain $1 < 2 < 3$, will be the initial data needed for the untwisted construction, and part of the data needed for the general construction.  Throughout we will denote an object of a $\mathsf 3$-parcel by the number which is its image under the functor $U$.

\section{The State-Sum Construction:  Untwisted Case}

We now describe a state-sum invariant of knot, Seifert surface, 3-manifold triples (which by general principles, cf. \cite{Y}, can be seen as arising from a generalized TQFT given by a functor on cobordisms of surfaces marked with curves with boundary), using only a finite $\mathsf 3$-parcel as initial data.

\begin{definition}
Let $\mathcal T$ be a flag-like triangulation of $K \subset \Sigma \subset M$ for a knot (or link) $K$ with a Seifert surface $\Sigma$ in a 3-manifold $M$.    The $\Gamma({\mathcal T})$, {\em directed graph} of $\mathcal T$, is the directed graph with ${\mathcal T}_0$, the set of vertices of $\mathcal T$ as vertices, and an edge $(v,w)$ whenever $v$ and $w$ are vertices connected by an edge of $\mathcal T$ and $\dim(v) \leq \dim(w)$, where the dimension of a vertex is the dimension of the stratum in which it lies.  ${\mathcal C}({\mathcal T})$, the category of $\mathcal T$, is then the quotient of the path category of $\Gamma({\mathcal T})$ by the relations $(v,w)(w,v) = Id_v$ whenever there are edges in both directions between two vertices, and $(v,w)(w,x) = (v,x)$ whenever there is a 2-simplex of $\mathcal T$ with vertices $v,w,$ and $x$.  ${\mathcal C}({\mathcal T})$ then has a surjective conservative functor $D$ to $\mathsf 3$, given on object by mapping each vertex to the dimension of the stratum in which it lies.
\end{definition}

\begin{definition} Given a $\mathsf 3$-parcel, $U:{\mathcal C}\rightarrow {\mathsf 3}$, a $\mathcal C$-coloring of a flag-like triangulation $\mathcal T$ of a knot-Seifert surface-3-manifold triple $K \subset \Sigma \subset M$ is a functor $\sigma:{\mathcal C}({\mathcal T}) \rightarrow {\mathcal C}$ such that $U(\sigma) = D$.
\end{definition}

\begin{theorem}
\label{untwisted} Let $K \subset \Sigma \subset M$ be a knot, Seifert surface, 3-manifold triple,  $\mathcal T$ a flag-like triangulation, and $U:{\mathcal C}\rightarrow {\mathsf 3}$ a $\mathsf 3$-parcel.  Let $[\mathcal T, \mathcal C]$ denote the set of $\mathcal C$-clorings of $\mathcal T$, and $G_i = {\mathcal C}(i,i)$ for $i = 1, 2,3$, denote the endomorphism groups of $\mathcal C$.  And finally let $\mathcal T_0^i$ for $i = 1,2,3$ denote the set of vertices of $\mathcal T$ lying in the stratum of dimension $i$.  Then, using $\#$ to denote cardinality, the quantity

\[ \frac{ \#[\mathcal T, \mathcal C] } { \prod_{i=1}^3 \#G_i^{\#T_0^i} } \]

is independent of the triangulation and thus gives a topological invariant of the triple.
\end{theorem}

\begin{proof}
We prove this theorem using Alexander moves.  The extended Pachner moves of section 2 will become important for the twisted case.

Whenever an Alexander subdivision is performed, a single new vertex is introduced.  Chosing an edge from the new vertex to an adjacent vertex in the same stratum induces an equivalence between the category of the old triangulation and the category of the new triangulation, by functors commuting with the underlying functors to $\mathsf 3$.  A $\mathcal C$-coloring of the old triangulation determines most of a $\mathcal C$-coloring of the new -- all that remains to be specified are the images of the arrows named by edges incident with the new vertex.  A brief consideration shows that the choice of an element of $G_i$ (where $i$ is the dimension of the stratum in which the new vertex lies) to be the image of the edge chosen to induce the equivalence of categories completely determines a unique $\mathcal G$-coloring of the new triangulation, and thus whenever a simplex in the $i$-dimensional stratum is subdivided, the number of $\mathcal G$-colorings of the new triangulation is exactly $\#G_i$ times the number of $\mathcal G$-colorings of the old triangulation, thus establishing the result.
\end{proof}

As with the invariants of $n$-manifolds given by untwisted Dikjgraaf-Witten theory, which for connected $n$-manifolds count group homomorphisms from the fundamental group to the guage group, and more generally count groupoid homomorphisms from any skeleton of the fundamental groupoid to the guage group, the invariants given by Theorem \ref{untwisted} admit a counting interpretation.

Observe that the relation $x \preceq y$ on points of $M$ given by $x \preceq y$ if $\dim(x) \leq \dim(y)$, where the dimension of a point is the dimension of the stratum in which it lies, is a preorder, whose restriction to each stratum is chaotic.  Taking the family of directed paths in $M$ to be the continuous maps $p:[0,1]\rightarrow M$ which are non-decreasing with respect to the usual order on $[0,1]$ and the preorder on $M$, $M$ is given the structure of a directed topological space, or $d$-space in the sense of Grandis \cite{Gr1, Gr2}.  It is easy to establish

\begin{theorem}  The fundamental category $\uparrow\!\!\Pi_1(X)$ with respect to the $d$-space structure of the previous paragraph has an underlying functor  to $\mathsf 3$ given on objects by $x \mapsto \dim(x)$, and $\uparrow\!\!\pi_1(X)$ is any skelton of $\uparrow\!\!\Pi_1(X)$, the invariant of Theorem \ref{untwisted} counts the number of functors from $\uparrow\!\!\pi_1(X)$ to $\mathcal C$, which commute with the underlying functors to $\mathsf 3$.
\end{theorem}

\section{Partial Cocycles}

We now turn to the matter of introducing local coefficients which generalize the group 3-cocycle in Wakui's construction \cite{W}.

The obvious generalization of a $K^\times$-valued group 3-cocycle to a category -- a function from composable triples of arrows in $\mathcal C$ to $K^\times$, satisfying the group 3-cocycle condition {\em mutatis mutandis} for every composable quadruple of arrows -- will give an instance of the rest of the necessary data.  However, a moment's thought reveals that the values of a cocycle on composable triples for which the target of the composition is the object $1$ or $2$ will be irrelevant to the construction, since they cannot occur as labels of a tetrahedron, and, in fact weaker conditions on the values on triples whose composition has target $3$ suffice.  The analogue of Wakui's 3-cocycles in our construction is given by a function on the composable triples of arrows which can occur as labels of the long oriented path in a tetrahedron, satisfying equations corresponding to the combinatorial moves that suffice to ensure topological invariance:

\begin{definition} A {\em $K^\times$-valued partial 3-cocycle} on a $\mathsf 3$-parcel $U:{\mathcal C}\rightarrow {\mathsf 3}$ is a function 
\[ \alpha: \{ (f,g,h) | f,g,h \in Arr({\mathcal C}), \; fgh \; \mbox{\rm is defined,}\;  t(h) = 3, \; \mbox{\rm and } s(h) \neq 1\; \} \rightarrow K^\times \]

\noindent for some field $K$, satisfying

\begin{enumerate}
\item $\alpha_{j,k,3,3}(g,h,k)\cdot\alpha_{i,k,3,3}^{-1}(fg,h,k)\cdot\alpha_{i,j,3,3}(f,gh,k)\cdot$ \\ \hspace*{1cm} $\alpha_{i,j,k,3}^{-1}(f,g,hk)\cdot\alpha_{i,j,k,3}(f,g,h) = 1$, whenever $k \geq 2$, and \\ \hspace*{1cm} $f,g,h,k$ is composable.

\item $ \alpha_{i , j , 2 , 3} ( a , b c , d )   \cdot  \alpha_{j , 2 , 2 , 3} ( b , c , d )  \cdot \alpha^{-1}_{i , j , 2 , 3} ( a , b , c d )   \cdot  \alpha^{-1}_{i , 2 , 2 , 3} ( a b , c , d )    = $ \\
\hspace*{1cm}$ \alpha_{i , j , 2 , 3} ( a , b c , e )  \cdot  \alpha_{j , 2 , 2 , 3} ( b , c , e ) \cdot  \alpha^{-1}_{i , j , 2 , 3} ( a , b , c e )  \cdot \alpha^{-1}_{i , 2 , 2 , 3} ( a b , c , e ), $\\ \hspace*{1cm} whenever $a,b,c,d$ and $a,b,c,e$ are both composable.

\item $\alpha_{1 , 1 , 2 , 3} ( a , b , c )  \cdot  \alpha^{-1}_{1 , 1 , 2 , 3} ( a , b , e )  \cdot  \alpha_{1 , 1 , 3 , 3} ( a , b c , d )  = $ \\
\hspace*{1cm}$\alpha_{1 , 1 , 2 , 3} ( f , g b , c )  \cdot \alpha^{-1}_{1 , 1 , 2 , 3} ( f , g b , e )  \cdot  \alpha_{1 , 1 , 3 , 3} ( f , g b c , d )  \cdot   $\\ \hspace*{1cm} $\alpha_{1 , 1 , 2 , 3} ( g , b , c )   \cdot  \alpha^{-1}_{1 , 1 , 2 , 3} ( g , b , e )   \cdot \alpha_{1 , 1 , 3 , 3} ( g , b c , d )  $, whenever \\ \hspace*{1cm}$a,b,c,d$ is composable, $e$ is any arrow such that $be = bcd$,\\ 
\hspace*{1cm}and $fg = a$. 
\end{enumerate}

\noindent where $\alpha_{i,j,k,3}$ denotes the restriction of $\alpha$ to $\{ (f,g,h) | s(f) = i, t(f) = j = s(g), t(g) = k = s(h), t(h) = 3 \}$ for $1 \leq i \leq j \leq k \leq3$, and the quadruples of maps in the condition have sources and targets agreeing with those specified by the subscripts on instances of $\alpha$.

\end{definition}

\begin{proposition} If $\beta:\{(f,g,h) | f,g,h \in Arr({\mathcal C}) \; \mbox{\rm and } fgh \;\mbox{\rm is defined} \}\rightarrow K^\times$ is a 3-cocycle on the category $\mathcal C$ of a $\mathsf 3$-parcel, then the restriction of $\beta$ to $\{ (f,g,h) | f,g,h \in Arr({\mathcal C}), \; fgh \; \mbox{\rm is defined,}\; t(h) = 3 \; \mbox{\rm and } s(h) \neq 1\}$ is a partial 3-cocycle.
\end{proposition}

\begin{proof}
The first condition in the definition of a partial cocycle follows trivially from the cocycle condition, being simply instances of the cocycle condition.  The second follows since in the presence of the cocycle condition each side equals $\beta(a, b, c)$.   

The third requires a little work.  Solving the equation to move all factors to the right hand side, gives the equivalent condition

\begin{eqnarray*} 1 & = & \alpha^{-1}_{1 , 1 , 2 , 3} ( a , b , c )  \cdot  \alpha_{1 , 1 , 2 , 3} ( a , b , e )  \cdot  \alpha^{-1}_{1 , 1 , 3 , 3} ( a , b c , d ) t\\
& & \alpha_{1 , 1 , 2 , 3} ( f , g b , c )  \cdot \alpha^{-1}_{1 , 1 , 2 , 3} ( f , g b , e )  \cdot  \alpha_{1 , 1 , 3 , 3} ( f , g b c , d )  \cdot  \\ 
& &\alpha_{1 , 1 , 2 , 3} ( g , b , c )   \cdot  \alpha^{-1}_{1 , 1 , 2 , 3} ( g , b , e )   \cdot \alpha_{1 , 1 , 3 , 3} ( g , b c , d ). \end{eqnarray*}

Now, suppressing the subscripts denoting sources and targets, and substituting the cocycle $\beta$ gives

\begin{eqnarray*} 1 & = & \beta^{-1} ( a , b , c )  \cdot  \beta ( a , b , e )  \cdot  \beta^{-1}( a , b c , d ) \cdot\\
& & \beta( f , g b , c )  \cdot \beta^{-1} ( f , g b , e )  \cdot  \beta ( f , g b c , d )  \cdot  \\ 
& &\beta( g , b , c )   \cdot  \beta^{-1} ( g , b , e )   \cdot \beta ( g , b c , d ). \end{eqnarray*}

\noindent Now, recalling the $fg = a$, the factors aligned in each column in the preceding expression are three factors of an instance of the coboundary of $\beta$, for the composable quadruples $f,g,b,c$;  $f,g,b,e$ and $f,g,bc,d$, respectively.  Replacing the factors in each column with the product of two factors equal to them by the cocycle condition gives 

\begin{eqnarray*} 1 & = & \beta( f , g , bc )  \cdot \beta^{-1} ( f , g , be )  \cdot  \beta ( f , g  , bcd ) \cdot \\
& & \beta^{-1} ( f, g, b )  \cdot  \beta ( f, g, b )  \cdot  \beta^{-1}(f, g, bc ) \end{eqnarray*}

\noindent the factors on the right hand side of which cancel in pairs, when the condition that $be = bcd$ is recalled, thus completing the proof.  \end{proof}

As usual, to twist by coefficients, we need to specify orientations on the edges of the triangulation, and in the proof of invariance show not only invariance under the combinatorial moves giving PL homoemorphism, for us the moves of the second statement of Theorem \ref{sufficient_moves}, but under changes to the edge-orientations.  The orientation on the knot and the non-invertibility conditions in the definition of the category ${\mathcal C}({\mathcal T})$ remove some of the need to make choices, but not all -- the edges of the knot are oriented in agreement with the orientation of the knot, and edges incident with the knot, but not lying in it, are all oriented away from the knot, and edges incident with the Seifert surface, but not lying in it, are all oriented away from the surface.   To orient the other edges, as usual chose a linear ordering of the vertices (in this case it suffices to chose separate linear orderings of the vertices in the interior of the Seifert surface and of the vertices in the complement of the closed Seifert surface) and orient the remaing edges from the earlier vertex to the later vertex.

Once the edges are oriented, each tetrahedron in the triangulation has a longest oriented path of edges, and a combinatorial orientation induced by the ordering of the vertices along the longest path.  We can now state our main theorem:

\begin{theorem} \label{main}
Let $U:{\mathcal C} \rightarrow {\mathsf 3}$ be a $\mathsf 3$-parcel, and $\alpha$ be a partial 3-cocycle on it.
Then for any knot, Seifert surface, 3-manifold triple $K \subset \Sigma \subset M$, the quantity given by

\[  \sum_{\lambda:{\mathcal C}({\mathcal T})\rightarrow {\mathcal C}} \frac{\prod_{\sigma \in {\mathcal T}_3} 
\alpha^{\epsilon(\sigma)}(\lambda(\sigma))} { \prod_{i=1}^3 \#G_i^{\#T_0^i} }\]

\noindent  where the summation ranges over all $\mathcal C$-colorings of $\mathcal T$,  $\lambda(\sigma)$ denotes the composable of arrows $\lambda$ assigns to the longest path of the tetrahedron $\sigma$ and $\epsilon(\sigma)$ is $+1$ if the combinatorial orientation of $\sigma$ agrees with the orientation of $M$, and $-1$ otherwise, for any flag-like triangulation $\mathcal T$ and a choice of ordering of the vertices off the knot, is independent of the triangulation and choice of ordering, and thus a topological invariant of the triple.
\end{theorem}

\begin{proof}
As in the case of Theorem \ref{untwisted}, moves which introduce a new vertex involve a choice of a label for one of the edges incident with the new vertex and lying in the same stratum as the new vertex, which suffices, in the presence of labels from the remaining edges not subdivided by the move to determine uniquely a $\mathcal C$-coloring of the new triangulation.  The denominator is thus multiplied by the number of summands in the numerator for the new triangulation.

It thus suffices to show that for any move which does not introduce a new vertex (the 2-3 Pachner moves and the 4-4 extended Pachner move) for any $\mathcal C$-coloring of the initial triangulation, there is a unique $\mathcal C$-coloring of the triangulation resulting from the move, agreeing with the given coloring on all edges not modified by the move and that the summand in the numerator for these colorings are equal, and similarly for moves which do introduce a new vertex (the 1-4 Pachner move, the 2-6 extended Pachner move and the 3-6 move on an edge of the knot) that given a $\mathcal C$-coloring of the initial triangulation and a choice of label for one of the new edges incident with the new vertex and lying in the same stratum, there is a unique $\mathcal C$-coloring of the new triangulation agreeding with the given coloring on all edges not modified by the move, having the chosen label on the new edge, and, moreover, that the summand in the numerator for the new triangulation for this coloring equals the summand in the numerator for the old triangulation of the given coloring.

Condition (3) is exactly the condition needed to ensure that each summand corresponding to labelings of the six-tetrahedon state in the 3-6 move are equal to the summand corresponding to the labeling of the three-tetrahedron state, as illustrated in Figures \ref{three_tetrahedra} and \ref{six_tetrahedra}.

\begin{figure}[p] \caption{Three Tetrahedron State of the 3-6 Move \label{three_tetrahedra} \vspace*{6pt}}
\begin{tabular}{ll}

$\displaystyle 
\begin{tikzpicture}
\fill[blue,blue, opacity=.5] (1.5,-1.8)--(1.9,-0.5)--(1.5,1.8);
        \node[circle, fill, inner sep=.9pt, outer sep=0pt] (A) at (1.5,1.8){};
	\node[circle, fill, inner sep=.9pt, outer sep=0pt] (B) at (0,0){};
	\node[circle, fill, inner sep=.9pt, outer sep=0pt] (C) at (1.9,-0.5){};
	\node[circle, fill, inner sep=.9pt, outer sep=0pt] (D) at (3,0.5){};
	\node[circle, fill, inner sep=.9pt, outer sep=0pt] (E) at (1.5,-1.8){};
		\begin{scope}
			\draw [dashed,midarrow={>}] (B)--(D) node[font=\tiny, midway, left]{$ d $};	
			\draw [midarrow={>}] (A)--(B) node[font=\tiny, midway, left]{$ a b c $};
			\draw [midarrow={>}] (C)--(B) node[font=\tiny, midway, below]{$ c $};
			\draw [midarrow={>}] (C)--(D) node[font=\tiny, midway, left]{$ c d $};
			\draw [midarrow={>}] (A)--(C) node[font=\tiny, midway, right]{$ a b $};
			\draw [midarrow={>}] (A)--(D) node[font=\tiny, midway, right]{\,$ a b c d $\,};
			\draw [midarrow={>}] (E)--(B) node[font=\tiny, midway, left]{$ b c $\,};
			\draw [midarrow={>}] (E)--(C) node[font=\tiny, midway, right] {$ b $\,};
			\draw [midarrow={>}] (E)--(D) node[font=\tiny, midway, right]{$ b c d $};
			\draw [dashed,thick, midarrow={>}] (1.5,1.8) -- (1.5,-0.2) node[font=\tiny, midway, left]{$ a $};
			\draw [dashed,thick] (1.5,-0.6) -- (1.5,-1.8);
			\fill[red, opacity=.5] (A) circle (3pt);
			\fill[red, opacity=.5] (E) circle (3pt);
			\fill[blue, opacity=.5] (C) circle (3pt);
			\draw[thick,red,red, opacity=.5] (A)--(E);
			\draw[thick,blue,blue, opacity=.5] (A)--(C);
			\draw[thick,blue,blue, opacity=.5] (E)--(C);			
		\end{scope}
\end{tikzpicture}$

&
 
$ \left\{ 
 \begin{array}{ll}

\begin{tikzpicture}
\fill[blue,blue, opacity=.5] (1.5,-1.8)--(1.9,-0.5)--(1.5,1.8);
        \node[circle, fill, inner sep=.9pt, outer sep=0pt] (A) at (1.5,1.8){};
	\node[circle, fill, inner sep=.9pt, outer sep=0pt] (B) at (0,0){};
	\node[circle, fill, inner sep=.9pt, outer sep=0pt] (C) at (1.9,-0.5){};
	\node[circle, fill, inner sep=.9pt, outer sep=0pt] (E) at (1.5,-1.8){};
		\begin{scope}
			\draw [midarrow={>}] (A)--(B) node[font=\tiny, midway, left]{$ a b c $};
			\draw [midarrow={>}] (C)--(B) node[font=\tiny, midway, below]{$ c $};
			\draw [midarrow={>}] (A)--(C) node[font=\tiny, midway, right]{$ a b $};
			\draw [midarrow={>}] (E)--(B) node[font=\tiny, midway, left]{$ b c $\,};
			\draw [midarrow={>}] (E)--(C) node[font=\tiny, midway, right] {$ b $\,};
			\draw [dashed,thick, midarrow={>}] (1.5,1.8) -- (1.5,-0.2) node[font=\tiny, midway, left]{$ a $};
			\draw [dashed,thick] (1.5,-0.6) -- (1.5,-1.8);
			\fill[red, opacity=.5] (A) circle (3pt);
			\fill[red, opacity=.5] (E) circle (3pt);
			\fill[blue, opacity=.5] (C) circle (3pt);
			\draw[thick,red,red, opacity=.5] (A)--(E);
			\draw[thick,blue,blue, opacity=.5] (A)--(C);
			\draw[thick,blue,blue, opacity=.5] (E)--(C);
		\end{scope}
\end{tikzpicture}

& \alpha_{1 , 1 , 2 , 3} ( a , b , c ) 

\\
\\

\begin{tikzpicture}
\fill[blue,blue, opacity=.5] (1.5,-1.8)--(1.9,-0.5)--(1.5,1.8);
	\node[circle, fill, inner sep=.9pt, outer sep=0pt] (A) at (1.5,1.8){};
	\node[circle, fill, inner sep=.9pt, outer sep=0pt] (C) at (1.9,-0.5){};
	\node[circle, fill, inner sep=.9pt, outer sep=0pt] (D) at (3,0.5){};
	\node[circle, fill, inner sep=.9pt, outer sep=0pt] (E) at (1.5,-1.8){};
		\begin{scope}
			\draw [midarrow={>}] (C)--(D) node[font=\tiny, midway, left]{$ c d $};
			\draw [midarrow={>}] (A)--(C) node[font=\tiny, midway, right]{$ a b $};
			\draw [midarrow={>}] (A)--(D) node[font=\tiny, midway, right]{\,$ a b c d $\,};
			\draw [midarrow={>}] (E)--(C) node[font=\tiny, midway, right] {$ b $\,};
			\draw [midarrow={>}] (E)--(D) node[font=\tiny, midway, right]{$ b c d $};
			\draw [midarrow={>}] (A)--(E) node[font=\tiny, midway, left]{$ a $};
			\fill[red, opacity=.5] (A) circle (3pt);
			\fill[red, opacity=.5] (E) circle (3pt);
			\fill[blue, opacity=.5] (C) circle (3pt);
			\draw[thick,red,red, opacity=.5] (A)--(E);
			\draw[thick,blue,blue, opacity=.5] (A)--(C);
			\draw[thick,blue,blue, opacity=.5] (E)--(C);
		\end{scope}
\end{tikzpicture}

& \displaystyle \alpha^{-1}_{1 , 1 , 2 , 3} ( a , b , c d )

\\ 
\\

\begin{tikzpicture}
         \node[circle, fill, inner sep=.9pt, outer sep=0pt] (A) at (1.5,1.8){};
	\node[circle, fill, inner sep=.9pt, outer sep=0pt] (B) at (0,0){};
	\node[circle, fill, inner sep=.9pt, outer sep=0pt] (D) at (3,0.5){};
	\node[circle, fill, inner sep=.9pt, outer sep=0pt] (E) at (1.5,-1.8){};
		\begin{scope}
			\draw [dashed,midarrow={>}] (B)--(D) node[font=\tiny, midway, left]{$ d $};	
			\draw [midarrow={>}] (A)--(B) node[font=\tiny, midway, left]{$ a b c $};
			\draw [midarrow={>}] (A)--(D) node[font=\tiny, midway, right]{\,$ a b c d $\,};
			\draw [midarrow={>}] (E)--(B) node[font=\tiny, midway, left]{$ b c $\,};
			\draw [midarrow={>}] (E)--(D) node[font=\tiny, midway, right]{$ b c d $};
			\draw [midarrow={>}] (A)--(E) node[font=\tiny, midway, right]{$ a $};
			\fill[red, opacity=.5] (A) circle (3pt);
			\fill[red, opacity=.5] (E) circle (3pt);
			\draw[thick,red,red, opacity=.5] (A)--(E);
		\end{scope}
\end{tikzpicture}

&  \alpha_{1 , 1 , 3 , 3} ( a , b c , d )

\\
\end{array} \right.$

\\
\\

\end{tabular}

\end{figure}
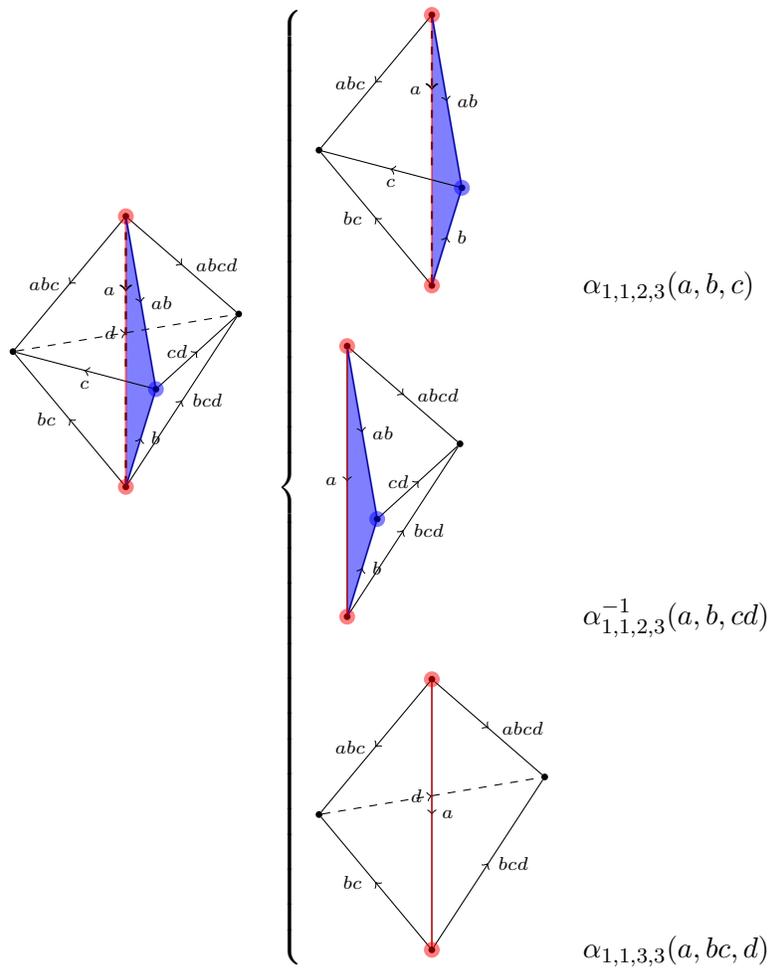

\begin{figure}[p]\caption{Six Tetrahedron State of the 3-6 Move \label{six_tetrahedra} \vspace*{6pt}}

\begin{tabular}{ll}

$\displaystyle 
\begin{tikzpicture}
\fill[blue,blue, opacity=.5] (1.5,-1.8)--(1.9,-0.5)--(1.5,1.8);
        \node[circle, fill, inner sep=.9pt, outer sep=0pt] (A) at (1.5,1.8){};
	\node[circle, fill, inner sep=.9pt, outer sep=0pt] (B) at (0,0){};
	\node[circle, fill, inner sep=.9pt, outer sep=0pt] (C) at (1.9,-0.5){};
	\node[circle, fill, inner sep=.9pt, outer sep=0pt] (D) at (3,0.5){};
	\node[circle, fill, inner sep=.9pt, outer sep=0pt] (E) at (1.5,-1.8){};
	\node[circle, fill, inner sep=.9pt, outer sep=0pt] (F) at (1.5,-0.1){};
		\begin{scope}
			\draw [dashed,midarrow={>}] (B)--(D) node[font=\tiny, midway, left]{$ d $};	
			\draw [midarrow={>}] (A)--(B) node[font=\tiny, midway, left]{$ a b c $};
			\draw [midarrow={>}] (C)--(B) node[font=\tiny, midway, below]{$ c $};
			\draw [midarrow={>}] (C)--(D) node[font=\tiny, midway, right]{$ c d $};
			\draw [midarrow={>}] (A)--(C) node[font=\tiny, midway, right]{$ a b $};
			\draw [midarrow={>}] (A)--(D) node[font=\tiny, midway, right]{\,$ a b c d $\,};
			\draw [midarrow={>}] (E)--(B) node[font=\tiny, midway, left]{$ b c $\,};
			\draw [midarrow={>}] (E)--(C) node[font=\tiny, midway, right] {$ b $\,};
			\draw [midarrow={>}] (E)--(D) node[font=\tiny, midway, right]{$ b c d $};
			\draw [dashed,thick, midarrow={>}] (1.5,1.8) -- (1.5,-0.2) node[font=\tiny, midway, left]{$ f $};
			\draw [dashed,thick, midarrow={>}] (1.5,-0.6) -- (1.5,-1.8) node[font=\tiny, midway, left]{$ g $};
			\fill[red, opacity=.5] (A) circle (3pt);
			\fill[red, opacity=.5] (E) circle (3pt);
			\fill[blue, opacity=.5] (C) circle (3pt);
			\draw[thick,red,red, opacity=.5] (A)--(E);
			\draw[thick,blue,blue, opacity=.5] (A)--(C);
			\draw[thick,blue,blue, opacity=.5] (E)--(C);	
			\draw [thick, midarrow={>}] (F)--(B) node[font=\tiny, midway, above]{$ g b c $};
			\draw [thick, midarrow={>}] (F)--(C) node[font=\tiny, midway, left]{$ g b $};
			\draw [thick, midarrow={>}] (F)--(D) node[font=\tiny, midway, above]{$ g b c d $};		
		\end{scope}
\end{tikzpicture}$

&
 
$ \left\{ 
 \begin{array}{ll}

\begin{tikzpicture}
\fill[blue,blue, opacity=.5] (1.5,0)--(1.9,-0.5)--(1.5,1.8);
       \node[circle, fill, inner sep=.9pt, outer sep=0pt] (A) at (1.5,1.8){};
	\node[circle, fill, inner sep=.9pt, outer sep=0pt] (B) at (0,0){};
	\node[circle, fill, inner sep=.9pt, outer sep=0pt] (C) at (1.9,-0.5){};
	\node[circle, fill, inner sep=.9pt, outer sep=0pt] (F) at (1.5,0){};
		\begin{scope}
			\draw [midarrow={>}] (A)--(B) node[font=\tiny, midway, left]{$ ( a b c = ) f g b c $};
			\draw [midarrow={>}] (C)--(B) node[font=\tiny, midway, below]{$ c $};
			\draw [midarrow={>}] (A)--(C) node[font=\tiny, midway, right]{$ a b ( = f g b )$};
			\draw [dashed,thick, midarrow={>}] (1.5,1.8) -- (1.5,0) node[font=\tiny, midway, left]{$ f $};
			\fill[red, opacity=.5] (A) circle (3pt);
			\fill[blue, opacity=.5] (C) circle (3pt);
			\draw[thick,red,red, opacity=.5] (A)--(1.5,0);
			\draw[thick,blue,blue, opacity=.5] (A)--(C);
			\draw [dashed, thick, midarrow={>}] (F)--(B) node[font=\tiny, midway, above]{$ g b c $};
			\draw [dashed, thick, midarrow={>}] (F)--(C) node[font=\tiny, midway, left]{$ g b $};
			\draw[thick,blue,blue, opacity=.5] (F)--(C);
			\fill[red, opacity=.5] (F) circle (3pt);
		\end{scope}
\end{tikzpicture}

& \alpha_{1 , 1 , 2 , 3} ( f , g b , c ) 

\\
\\

\begin{tikzpicture}
        \node[circle, fill, inner sep=.9pt, outer sep=0pt] (A) at (1.5,1.8){};
	\node[circle, fill, inner sep=.9pt, outer sep=0pt] (B) at (0,0){};
	\node[circle, fill, inner sep=.9pt, outer sep=0pt] (D) at (3,0.5){};
	\node[circle, fill, inner sep=.9pt, outer sep=0pt] (F) at (1.5,-0.1){};
		\begin{scope}
			\draw [dashed,midarrow={>}] (B)--(D) node[font=\tiny, midway, left]{$ d $};	
			\draw [midarrow={>}] (A)--(B) node[font=\tiny, midway, left]{$ ( a b c = ) f g b c $};
			\draw [midarrow={>}] (A)--(D) node[font=\tiny, midway, right]{\,$ f g b c d ( = a b c d ) $\,};
			\draw [thick, midarrow={>}] (1.5,1.8) -- (F) node[font=\tiny, midway, left]{$ f $};
			\fill[red, opacity=.5] (A) circle (3pt);
			\draw[thick,red,red, opacity=.5] (A)--(F);
			\draw [thick, midarrow={>}] (F)--(B) node[font=\tiny, midway, below]{$ g b c $};
			\draw [thick, midarrow={>}] (F)--(D) node[font=\tiny, midway, below]{$ g b c d $};
			\fill[red, opacity=.5] (F) circle (3pt);
		\end{scope}
\end{tikzpicture}

& \displaystyle \alpha_{1 , 1 , 3 , 3} ( f , g b c , d )

\\ 
\\

\begin{tikzpicture}
\fill[blue,blue, opacity=.5] (1.5,-0.1)--(1.9,-0.5)--(1.5,1.8);
        \node[circle, fill, inner sep=.9pt, outer sep=0pt] (A) at (1.5,1.8){};
	\node[circle, fill, inner sep=.9pt, outer sep=0pt] (C) at (1.9,-0.5){};
	\node[circle, fill, inner sep=.9pt, outer sep=0pt] (D) at (3,0.5){};
	\node[circle, fill, inner sep=.9pt, outer sep=0pt] (F) at (1.5,-0.1){};
		\begin{scope}
			\draw [midarrow={>}] (C)--(D) node[font=\tiny, midway, below]{$ c d $};
			\draw [midarrow={>}] (A)--(C) node[font=\tiny, midway, right]{$ f g b(=\!a b ) $};
			\draw [midarrow={>}] (A)--(D) node[font=\tiny, midway, right]{\,$ f g b c d ( = a b c d ) $\,};
			\draw [dashed,thick, midarrow={>}] (1.5,1.8) -- (1.5,-0.2) node[font=\tiny, midway, left]{$ f $};
			\fill[red, opacity=.5] (A) circle (3pt);
			\fill[blue, opacity=.5] (C) circle (3pt);
			\draw[thick,red,red, opacity=.5] (A)--(F);
			\draw[thick,blue,blue, opacity=.5] (A)--(C);
			\draw [thick, midarrow={>}] (F)--(C) node[font=\tiny, midway, left]{$ g b $};
			\draw [dashed, thick, midarrow={>}] (F)--(D) node[font=\tiny, midway, right]{$ g b c d $};
			\draw[thick,blue,blue, opacity=.5] (F)--(C);
			\fill[red, opacity=.5] (F) circle (3pt);
		\end{scope}
\end{tikzpicture}

&  \alpha^{-1}_{1 , 1 , 2 , 3} ( f , g b , c d )

\\
\\

\begin{tikzpicture}
\fill[blue,blue, opacity=.5] (1.5,-1.8)--(1.9,-0.5)--(1.5,0);
	\node[circle, fill, inner sep=.9pt, outer sep=0pt] (B) at (0,0){};
	\node[circle, fill, inner sep=.9pt, outer sep=0pt] (C) at (1.9,-0.5){};
	\node[circle, fill, inner sep=.9pt, outer sep=0pt] (E) at (1.5,-1.8){};
	\node[circle, fill, inner sep=.9pt, outer sep=0pt] (F) at (1.5,0){};
		\begin{scope}
			\draw [midarrow={>}] (C)--(B) node[font=\tiny, midway, below]{$ c $};
			\draw [midarrow={>}] (E)--(B) node[font=\tiny, midway, left]{$ b c $\,};
			\draw [midarrow={>}] (E)--(C) node[font=\tiny, midway, right] {$ b $\,};
			\draw [dashed,thick, midarrow={>}] (F) -- (1.5,-1.8) node[font=\tiny, midway, left]{$ g $};
			\fill[red, opacity=.5] (E) circle (3pt);
			\fill[blue, opacity=.5] (C) circle (3pt);
			\draw[thick,blue,blue, opacity=.5] (E)--(C);	
			\draw [thick, midarrow={>}] (F)--(B) node[font=\tiny, midway, above]{$ g b c $};
			\draw [thick, midarrow={>}] (F)--(C) node[font=\tiny, midway, above]{\,\,\,\,$ g b $};
			\fill[red, opacity=.5] (F) circle (3pt);
			\draw[thick,red,red, opacity=.5] (F)--(E);
			\draw[thick,blue,blue, opacity=.5] (F)--(C);	
		\end{scope}
\end{tikzpicture}

&  \alpha_{1 , 1 , 2 , 3} ( g , b , c )

\\
\\

\begin{tikzpicture}
	\node[circle, fill, inner sep=.9pt, outer sep=0pt] (B) at (0,0){};
	\node[circle, fill, inner sep=.9pt, outer sep=0pt] (D) at (3,0.5){};
	\node[circle, fill, inner sep=.9pt, outer sep=0pt] (E) at (1.5,-1.8){};
	\node[circle, fill, inner sep=.9pt, outer sep=0pt] (F) at (1.5,-0.1){};
		\begin{scope}
			\draw [midarrow={>}] (B)--(D) node[font=\tiny, midway, above]{$ d $};	
			\draw [midarrow={>}] (E)--(B) node[font=\tiny, midway, left]{$ b c $\,};
			\draw [midarrow={>}] (E)--(D) node[font=\tiny, midway, right]{$ b c d $};
			\draw [thick, midarrow={>}] (F) -- (1.5,-1.8) node[font=\tiny, midway, left]{$ g $};
			\fill[red, opacity=.5] (F) circle (3pt);
			\fill[red, opacity=.5] (E) circle (3pt);
			\draw[thick,red,red, opacity=.5] (F)--(E);
			\draw [thick, midarrow={>}] (F)--(B) node[font=\tiny, midway, below]{$ g b c $};
			\draw [thick, midarrow={>}] (F)--(D) node[font=\tiny, midway, below]{$ g b c d $};
		\end{scope}
\end{tikzpicture}

&  \alpha_{1 , 1 , 3 , 3} ( g , b c , d )

\\
\\

\begin{tikzpicture}
\fill[blue,blue, opacity=.5] (1.5,-1.8)--(1.9,-0.5)--(1.5,-0.1);
	\node[circle, fill, inner sep=.9pt, outer sep=0pt] (C) at (1.9,-0.5){};
	\node[circle, fill, inner sep=.9pt, outer sep=0pt] (D) at (3,0.5){};
	\node[circle, fill, inner sep=.9pt, outer sep=0pt] (E) at (1.5,-1.8){};
	\node[circle, fill, inner sep=.9pt, outer sep=0pt] (F) at (1.5,-0.1){};
		\begin{scope}
			\draw [midarrow={>}] (C)--(D) node[font=\tiny, midway, right]{$ c d $};
			\draw [midarrow={>}] (E)--(C) node[font=\tiny, midway, right] {\!\!$ b $\,};
			\draw [midarrow={>}] (E)--(D) node[font=\tiny, midway, right]{$ b c d $};
			\draw [thick, midarrow={>}] (F) -- (1.5,-1.8) node[font=\tiny, midway, left]{$ g $};
			\fill[red, opacity=.5] (F) circle (3pt);
			\fill[red, opacity=.5] (E) circle (3pt);
			\fill[blue, opacity=.5] (C) circle (3pt);
			\draw[thick,red,red, opacity=.5] (F)--(E);
			\draw[thick,blue,blue, opacity=.5] (E)--(C);	
			\draw[thick,blue,blue, opacity=.5] (F)--(C);	
			\draw [thick, midarrow={>}] (F)--(C) node[font=\tiny, midway, below]{$ g b $};
			\draw [thick, midarrow={>}] (F)--(D) node[font=\tiny, midway, above]{$ g b c d $};
		\end{scope}
\end{tikzpicture}

&  \alpha^{-1}_{1 , 1 , 2 , 3} ( g , b , c d )

\\

\end{array} \right.$

\\
\\

\end{tabular}
\end{figure}

Conditions (1) and (2) each ensure invariance under two types of moves -- the 1-4 and 2-3 Pachner moves and the 2-6 and 4-4 extended Pachner moves, respectively.  In each case, the differently indexed instances of the conditions correspond to the ways in which a flag-like triangulation of the state with fewer tetrahedra can intersect the strata.

For instance, the cases of the 4-4 move and the 2-6 move in which the tetrahedra being modified by the moves are not incident with the knot are show in Figures \ref{four_before} and \ref{four_after} and Figures \ref{two_extended} and \ref{six_no_knot}, respectively.  The latter depends on the choice of the label $f$, which then induces the labels $g$ and $h$ such that there are factorizations of $c$ and $d$ as $c = fg$ and $d = fh$, respectively.


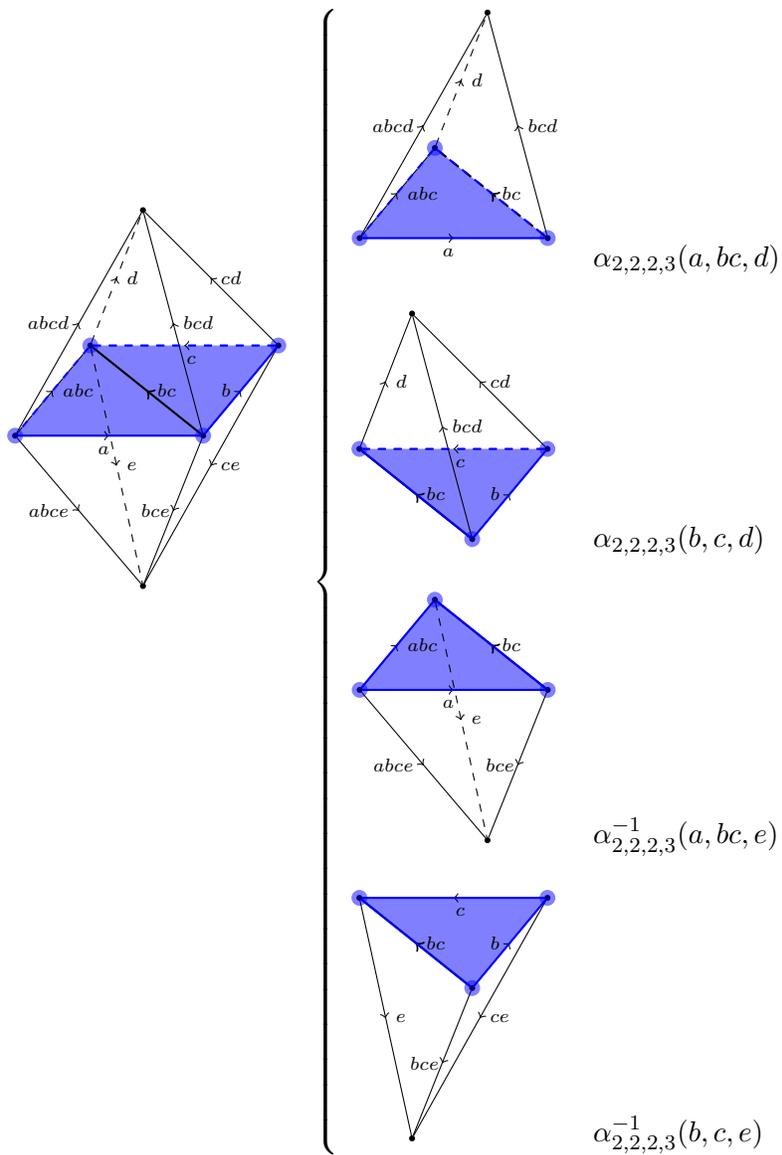
\begin{figure}[p]\caption{Initial State of 4-4 Move \label{four_before} \vspace*{6pt}}
\begin{tabular}{ll}

$\displaystyle
\begin{tikzpicture}
	\fill[blue,blue, opacity=.5] (0,0)--(2.5,0)--(3.5,1.2)--(1,1.2)--(0,0);
 	\node[circle, fill, inner sep=.8pt, outer sep=0pt] (A) at (0,0){};
	\node[circle, fill, inner sep=.8pt, outer sep=0pt] (B) at (2.5,0){};
	\node[circle, fill, inner sep=.8pt, outer sep=0pt] (C) at (3.5,1.2){};
	\node[circle, fill, inner sep=.8pt, outer sep=0pt] (D) at (1,1.2){};
	\node[circle, fill, inner sep=.8pt, outer sep=0pt] (E) at (1.7,3){};
	\node[circle, fill, inner sep=.8pt, outer sep=0pt] (F) at (1.7,-2){};
	\begin{scope}
		\draw [midarrow={>}] (A)--(B) node[font=\tiny, midway, below]{$ a $\,\,\,};
		\draw [midarrow={>}] (B)--(C) node[font=\tiny, midway, left]{$ b $};
		\draw [dashed,midarrow={>}] (C)--(D) node[font=\tiny, midway, below]{\,\,\,\,$ c $};
		\draw [dashed,midarrow={>}] (A)--(D) node[font=\tiny, midway, right]{$ abc $};
		\draw [thick,midarrow={>}] (B)--(D) node[font=\tiny, midway, right]{$ bc $};
		\draw [midarrow={>}] (A)--(E) node[font=\tiny, midway, left]{$ abcd $};
		\draw [midarrow={>}] (B)--(E) node[font=\tiny, midway, right]{$ bcd $};
		\draw [midarrow={>}] (C)--(E) node[font=\tiny, midway, right]{$ cd $};
		\draw [dashed,midarrow={>}] (D)--(E) node[font=\tiny, midway, right]{$ d $};
		\draw [midarrow={>}] (A)--(F) node[font=\tiny, midway, left]{$ abce $};
		\draw [midarrow={>}] (B)--(F) node[font=\tiny, midway, left]{$ bce$\!\!};
		\draw [midarrow={>}] (C)--(F) node[font=\tiny, midway, right]{$ ce $};
		\draw [dashed,midarrow={>}] (D)--(F) node[font=\tiny, midway, right]{$ e $};
		\draw[thick,blue,blue] (A)--(B)--(C);
		\draw[dashed,thick,blue,blue] (C)--(D)--(A);
		\fill[blue, opacity=.5] (A) circle (3pt);
		\fill[blue, opacity=.5] (B) circle (3pt);
		\fill[blue, opacity=.5] (C) circle (3pt);
		\fill[blue, opacity=.5] (D) circle (3pt);
	\end{scope}
\end{tikzpicture}$

&
 
$ \left\{ 
 \begin{array}{ll}

\begin{tikzpicture}
	\fill[blue,blue, opacity=.5] (0,0)--(2.5,0)--(1,1.2)--(0,0);
 	\node[circle, fill, inner sep=.8pt, outer sep=0pt] (A) at (0,0){};
	\node[circle, fill, inner sep=.8pt, outer sep=0pt] (B) at (2.5,0){};
	\node[circle, fill, inner sep=.8pt, outer sep=0pt] (D) at (1,1.2){};
	\node[circle, fill, inner sep=.8pt, outer sep=0pt] (E) at (1.7,3){};
	\begin{scope}
		\draw [midarrow={>}] (A)--(B) node[font=\tiny, midway, below]{$ a $\,\,\,};
		\draw [dashed,midarrow={>}] (A)--(D) node[font=\tiny, midway, right]{$ abc $};
		\draw [dashed,thick,midarrow={>}] (B)--(D) node[font=\tiny, midway, right]{$ bc $};
		\draw [midarrow={>}] (A)--(E) node[font=\tiny, midway, left]{$ abcd $};
		\draw [midarrow={>}] (B)--(E) node[font=\tiny, midway, right]{$ bcd $};
		\draw [dashed,midarrow={>}] (D)--(E) node[font=\tiny, midway, right]{$ d $};
		\draw[thick,blue,blue] (A)--(B);
		\draw[dashed,thick,blue,blue] (D)--(A);
		\draw[dashed,thick,blue,blue] (D)--(B);
		\fill[blue, opacity=.5] (A) circle (3pt);
		\fill[blue, opacity=.5] (B) circle (3pt);
		\fill[blue, opacity=.5] (D) circle (3pt);
	\end{scope}
\end{tikzpicture}

& \alpha_{2 , 2 , 2 , 3} ( a , b c , d ) 
\\
\\

\begin{tikzpicture}
	\fill[blue,blue, opacity=.5] (2.5,0)--(3.5,1.2)--(1,1.2);
	\node[circle, fill, inner sep=.8pt, outer sep=0pt] (B) at (2.5,0){};
	\node[circle, fill, inner sep=.8pt, outer sep=0pt] (C) at (3.5,1.2){};
	\node[circle, fill, inner sep=.8pt, outer sep=0pt] (D) at (1,1.2){};
	\node[circle, fill, inner sep=.8pt, outer sep=0pt] (E) at (1.7,3){};
	\begin{scope}
		\draw [midarrow={>}] (B)--(C) node[font=\tiny, midway, left]{$ b $};
		\draw [dashed,midarrow={>}] (C)--(D) node[font=\tiny, midway, below]{\,\,\,\,$ c $};
		\draw [thick,midarrow={>}] (B)--(D) node[font=\tiny, midway, right]{$ bc $};
		\draw [midarrow={>}] (B)--(E) node[font=\tiny, midway, right]{$ bcd $};
		\draw [midarrow={>}] (C)--(E) node[font=\tiny, midway, right]{$ cd $};
		\draw [midarrow={>}] (D)--(E) node[font=\tiny, midway, right]{$ d $};
		\draw[thick,blue,blue] (B)--(C);
		\draw[thick,blue,blue] (B)--(D);
		\draw[dashed,thick,blue,blue] (C)--(D);
		\fill[blue, opacity=.5] (B) circle (3pt);
		\fill[blue, opacity=.5] (C) circle (3pt);
		\fill[blue, opacity=.5] (D) circle (3pt);
	\end{scope}
\end{tikzpicture}

& \displaystyle \alpha_{2 , 2 , 2 , 3} ( b , c , d )
\\
\\

\begin{tikzpicture}
	\fill[blue,blue, opacity=.5] (0,0)--(2.5,0)--(1,1.2)--(0,0);
 	\node[circle, fill, inner sep=.8pt, outer sep=0pt] (A) at (0,0){};
	\node[circle, fill, inner sep=.8pt, outer sep=0pt] (B) at (2.5,0){};
	\node[circle, fill, inner sep=.8pt, outer sep=0pt] (D) at (1,1.2){};
	\node[circle, fill, inner sep=.8pt, outer sep=0pt] (F) at (1.7,-2){};
	\begin{scope}
		\draw [midarrow={>}] (A)--(B) node[font=\tiny, midway, below]{$ a $\,\,\,};
		\draw [midarrow={>}] (A)--(D) node[font=\tiny, midway, right]{$ abc $};
		\draw [thick,midarrow={>}] (B)--(D) node[font=\tiny, midway, right]{$ bc $};
		\draw [midarrow={>}] (A)--(F) node[font=\tiny, midway, left]{$ abce $};
		\draw [midarrow={>}] (B)--(F) node[font=\tiny, midway, left]{$ bce$\!\!};
		\draw [dashed,midarrow={>}] (D)--(F) node[font=\tiny, midway, right]{$ e $};
		\draw[thick,blue,blue] (A)--(B);
		\draw[thick,blue,blue] (D)--(A);
		\draw[thick,blue,blue] (D)--(B);
		\fill[blue, opacity=.5] (A) circle (3pt);
		\fill[blue, opacity=.5] (B) circle (3pt);
		\fill[blue, opacity=.5] (D) circle (3pt);
	\end{scope}
\end{tikzpicture}

& \displaystyle \alpha^{-1}_{2 , 2 , 2 , 3} ( a , b c , e )
\\
\\

\begin{tikzpicture}
	\fill[blue,blue, opacity=.5] (2.5,0)--(3.5,1.2)--(1,1.2);
	\node[circle, fill, inner sep=.8pt, outer sep=0pt] (B) at (2.5,0){};
	\node[circle, fill, inner sep=.8pt, outer sep=0pt] (C) at (3.5,1.2){};
	\node[circle, fill, inner sep=.8pt, outer sep=0pt] (D) at (1,1.2){};
	\node[circle, fill, inner sep=.8pt, outer sep=0pt] (F) at (1.7,-2){};
	\begin{scope}
		\draw [midarrow={>}] (B)--(C) node[font=\tiny, midway, left]{$ b $};
		\draw [midarrow={>}] (C)--(D) node[font=\tiny, midway, below]{\,\,\,\,$ c $};
		\draw [thick,midarrow={>}] (B)--(D) node[font=\tiny, midway, right]{$ bc $};
		\draw [midarrow={>}] (B)--(F) node[font=\tiny, midway, left]{$ bce$\!\!};
		\draw [midarrow={>}] (C)--(F) node[font=\tiny, midway, right]{$ ce $};
		\draw [midarrow={>}] (D)--(F) node[font=\tiny, midway, right]{$ e $};
		\draw[thick,blue,blue] (B)--(C);
		\draw[thick,blue,blue] (C)--(D);
		\draw[thick,blue,blue] (B)--(D);
		\fill[blue, opacity=.5] (B) circle (3pt);
		\fill[blue, opacity=.5] (C) circle (3pt);
		\fill[blue, opacity=.5] (D) circle (3pt);
	\end{scope}
\end{tikzpicture}

& \displaystyle \alpha^{-1}_{2 , 2 , 2 , 3} ( b , c , e )
\\

\end{array} \right.$

\end{tabular}
\end{figure}

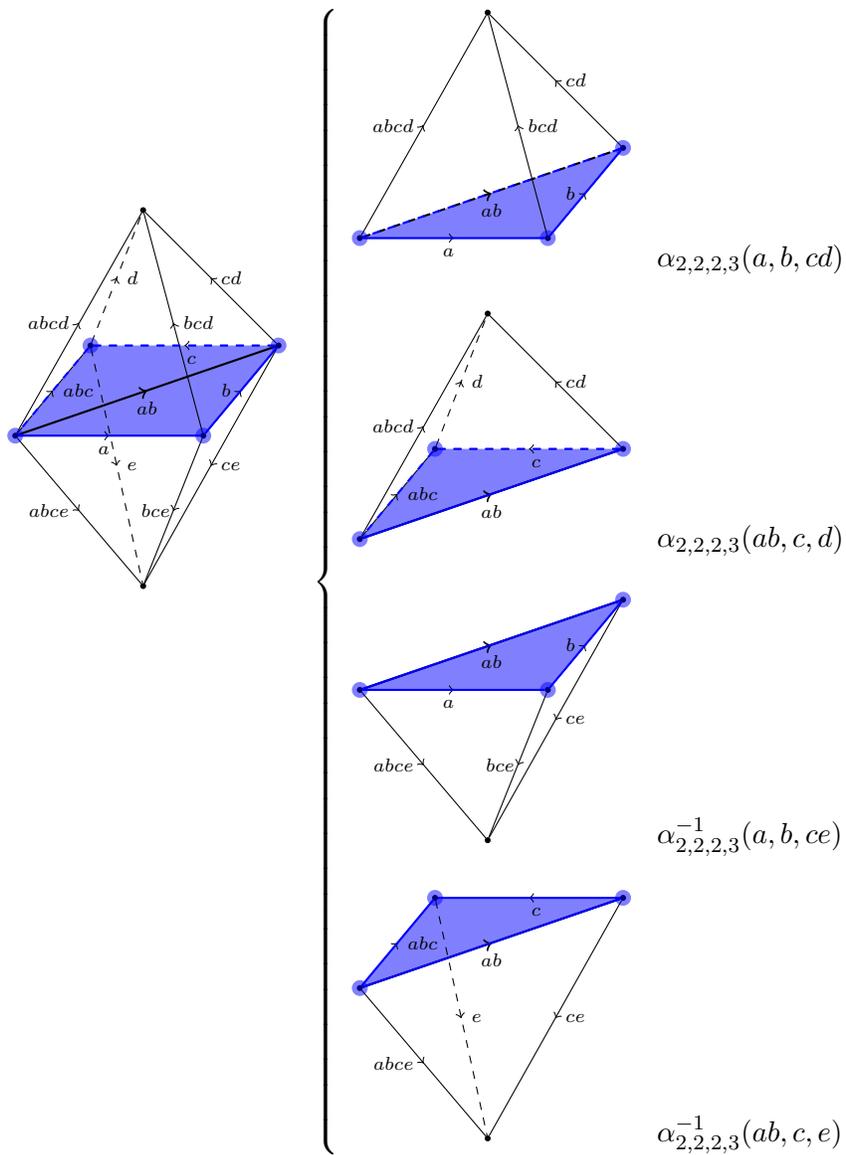
\begin{figure}[p]\caption{Final State of 4-4 Move \label{four_after} \vspace*{6pt}}
\begin{tabular}{ll}

$\displaystyle 
\begin{tikzpicture}
	\fill[blue,blue, opacity=.5] (0,0)--(2.5,0)--(3.5,1.2)--(1,1.2)--(0,0);
 	\node[circle, fill, inner sep=.8pt, outer sep=0pt] (A) at (0,0){};
	\node[circle, fill, inner sep=.8pt, outer sep=0pt] (B) at (2.5,0){};
	\node[circle, fill, inner sep=.8pt, outer sep=0pt] (C) at (3.5,1.2){};
	\node[circle, fill, inner sep=.8pt, outer sep=0pt] (D) at (1,1.2){};
	\node[circle, fill, inner sep=.8pt, outer sep=0pt] (E) at (1.7,3){};
	\node[circle, fill, inner sep=.8pt, outer sep=0pt] (F) at (1.7,-2){};
	\begin{scope}
		\draw [midarrow={>}] (A)--(B) node[font=\tiny, midway, below]{$ a $\,\,\,};
		\draw [midarrow={>}] (B)--(C) node[font=\tiny, midway, left]{$ b $};
		\draw [dashed,midarrow={>}] (C)--(D) node[font=\tiny, midway, below]{\,\,\,\,$ c $};
		\draw [dashed,midarrow={>}] (A)--(D) node[font=\tiny, midway, right]{$ abc $};
		\draw [thick,midarrow={>}] (A)--(C) node[font=\tiny, midway, below]{$ ab $};
		\draw [midarrow={>}] (A)--(E) node[font=\tiny, midway, left]{$ abcd $};
		\draw [midarrow={>}] (B)--(E) node[font=\tiny, midway, right]{$ bcd $};
		\draw [midarrow={>}] (C)--(E) node[font=\tiny, midway, right]{$ cd $};
		\draw [dashed,midarrow={>}] (D)--(E) node[font=\tiny, midway, right]{$ d $};
		\draw [midarrow={>}] (A)--(F) node[font=\tiny, midway, left]{$ abce $};
		\draw [midarrow={>}] (B)--(F) node[font=\tiny, midway, left]{$ bce$\!\!};
		\draw [midarrow={>}] (C)--(F) node[font=\tiny, midway, right]{$ ce $};
		\draw [dashed,midarrow={>}] (D)--(F) node[font=\tiny, midway, right]{$ e $};
		\draw[thick,blue,blue] (A)--(B)--(C);
		\draw[dashed,thick,blue,blue] (C)--(D)--(A);
		\fill[blue, opacity=.5] (A) circle (3pt);
		\fill[blue, opacity=.5] (B) circle (3pt);
		\fill[blue, opacity=.5] (C) circle (3pt);
		\fill[blue, opacity=.5] (D) circle (3pt);
	\end{scope}
\end{tikzpicture}$

&
 
$ \left\{ 
 \begin{array}{ll}

\begin{tikzpicture}
	\fill[blue,blue, opacity=.5] (0,0)--(2.5,0)--(3.5,1.2)--(0,0);
 	\node[circle, fill, inner sep=.8pt, outer sep=0pt] (A) at (0,0){};
	\node[circle, fill, inner sep=.8pt, outer sep=0pt] (B) at (2.5,0){};
	\node[circle, fill, inner sep=.8pt, outer sep=0pt] (C) at (3.5,1.2){};
	\node[circle, fill, inner sep=.8pt, outer sep=0pt] (E) at (1.7,3){};
	\begin{scope}
		\draw [midarrow={>}] (A)--(B) node[font=\tiny, midway, below]{$ a $\,\,\,};
		\draw [midarrow={>}] (B)--(C) node[font=\tiny, midway, left]{$ b $};
		\draw [dashed,thick,midarrow={>}] (A)--(C) node[font=\tiny, midway, below]{$ ab $};
		\draw [midarrow={>}] (A)--(E) node[font=\tiny, midway, left]{$ abcd $};
		\draw [midarrow={>}] (B)--(E) node[font=\tiny, midway, right]{$ bcd $};
		\draw [midarrow={>}] (C)--(E) node[font=\tiny, midway, right]{$ cd $};
		\draw[thick,blue,blue] (A)--(B)--(C);
		\draw[dashed,thick,blue,blue] (C)--(A);
		\fill[blue, opacity=.5] (A) circle (3pt);
		\fill[blue, opacity=.5] (B) circle (3pt);
		\fill[blue, opacity=.5] (C) circle (3pt);
	\end{scope}
\end{tikzpicture}

& \alpha_{2 , 2 , 2 , 3} ( a , b , c d ) 
\\
\\

\begin{tikzpicture}
	\fill[blue,blue, opacity=.5] (0,0)--(3.5,1.2)--(1,1.2)--(0,0);
 	\node[circle, fill, inner sep=.8pt, outer sep=0pt] (A) at (0,0){};
	\node[circle, fill, inner sep=.8pt, outer sep=0pt] (C) at (3.5,1.2){};
	\node[circle, fill, inner sep=.8pt, outer sep=0pt] (D) at (1,1.2){};
	\node[circle, fill, inner sep=.8pt, outer sep=0pt] (E) at (1.7,3){};
	\begin{scope}
		\draw [dashed,midarrow={>}] (C)--(D) node[font=\tiny, midway, below]{\,\,\,\,$ c $};
		\draw [dashed,midarrow={>}] (A)--(D) node[font=\tiny, midway, right]{$ abc $};
		\draw [thick,midarrow={>}] (A)--(C) node[font=\tiny, midway, below]{$ ab $};
		\draw [midarrow={>}] (A)--(E) node[font=\tiny, midway, left]{$ abcd $};
		\draw [midarrow={>}] (C)--(E) node[font=\tiny, midway, right]{$ cd $};
		\draw [dashed,midarrow={>}] (D)--(E) node[font=\tiny, midway, right]{$ d $};
		\draw[thick,blue,blue] (A)--(C);
		\draw[dashed,thick,blue,blue] (C)--(D)--(A);
		\fill[blue, opacity=.5] (A) circle (3pt);
		\fill[blue, opacity=.5] (C) circle (3pt);
		\fill[blue, opacity=.5] (D) circle (3pt);
	\end{scope}
\end{tikzpicture}

& \displaystyle \alpha_{2 , 2 , 2 , 3} ( a b , c , d )
\\
\\

\begin{tikzpicture}
	\fill[blue,blue, opacity=.5] (0,0)--(2.5,0)--(3.5,1.2)--(0,0);
 	\node[circle, fill, inner sep=.8pt, outer sep=0pt] (A) at (0,0){};
	\node[circle, fill, inner sep=.8pt, outer sep=0pt] (B) at (2.5,0){};
	\node[circle, fill, inner sep=.8pt, outer sep=0pt] (C) at (3.5,1.2){};
	\node[circle, fill, inner sep=.8pt, outer sep=0pt] (F) at (1.7,-2){};
	\begin{scope}
		\draw [midarrow={>}] (A)--(B) node[font=\tiny, midway, below]{$ a $\,\,\,};
		\draw [midarrow={>}] (B)--(C) node[font=\tiny, midway, left]{$ b $};
		\draw [thick,midarrow={>}] (A)--(C) node[font=\tiny, midway, below]{$ ab $};
		\draw [midarrow={>}] (A)--(F) node[font=\tiny, midway, left]{$ abce $};
		\draw [midarrow={>}] (B)--(F) node[font=\tiny, midway, left]{$ bce$\!\!};
		\draw [midarrow={>}] (C)--(F) node[font=\tiny, midway, right]{$ ce $};
		\draw[thick,blue,blue] (A)--(B)--(C);
		\draw[thick,blue,blue] (C)--(A);
		\fill[blue, opacity=.5] (A) circle (3pt);
		\fill[blue, opacity=.5] (B) circle (3pt);
		\fill[blue, opacity=.5] (C) circle (3pt);
	\end{scope}
\end{tikzpicture}

& \displaystyle \alpha^{-1}_{2 , 2 , 2 , 3} ( a , b , c e )
\\
\\

\begin{tikzpicture}
	\fill[blue,blue, opacity=.5] (0,0)--(3.5,1.2)--(1,1.2)--(0,0);
 	\node[circle, fill, inner sep=.8pt, outer sep=0pt] (A) at (0,0){};
	\node[circle, fill, inner sep=.8pt, outer sep=0pt] (C) at (3.5,1.2){};
	\node[circle, fill, inner sep=.8pt, outer sep=0pt] (D) at (1,1.2){};
	\node[circle, fill, inner sep=.8pt, outer sep=0pt] (F) at (1.7,-2){};
	\begin{scope}
		\draw [dashed,midarrow={>}] (C)--(D) node[font=\tiny, midway, below]{\,\,\,\,$ c $};
		\draw [dashed,midarrow={>}] (A)--(D) node[font=\tiny, midway, right]{$ abc $};
		\draw [thick,midarrow={>}] (A)--(C) node[font=\tiny, midway, below]{$ ab $};
		\draw [midarrow={>}] (A)--(F) node[font=\tiny, midway, left]{$ abce $};
		\draw [midarrow={>}] (C)--(F) node[font=\tiny, midway, right]{$ ce $};
		\draw [dashed,midarrow={>}] (D)--(F) node[font=\tiny, midway, right]{$ e $};
		\draw[thick,blue,blue] (A)--(C);
		\draw[thick,blue,blue] (C)--(D)--(A);
		\fill[blue, opacity=.5] (A) circle (3pt);
		\fill[blue, opacity=.5] (C) circle (3pt);
		\fill[blue, opacity=.5] (D) circle (3pt);
	\end{scope}
\end{tikzpicture}

& \displaystyle \alpha^{-1}_{2 , 2 , 2 , 3} ( a b , c , e )
\\

\end{array} \right.$

\end{tabular}
\end{figure}



\begin{figure}[p]\caption{Two Tetrahedron State of 2-6 Move \label{two_extended} \vspace*{6pt}}
\begin{tabular}{ll}

$\displaystyle 
\begin{tikzpicture}
\fill[blue,blue, opacity=.5] (0,0)--(1.9,-0.5)--(3,0.5)--(0,0);
	\node[circle, fill, inner sep=.9pt, outer sep=0pt] (A) at (0,0){};
	\node[circle, fill, inner sep=.9pt, outer sep=0pt] (B) at (1.9,-0.5){};
	\node[circle, fill, inner sep=.9pt, outer sep=0pt] (C) at (3,0.5){};
	\node[circle, fill, inner sep=.9pt, outer sep=0pt] (D) at (1.5,1.8){};
	\node[circle, fill, inner sep=.9pt, outer sep=0pt] (E) at (1.5,-1.8){};
		\begin{scope}
			\draw [midarrow={>}] (A)--(B) node[font=\tiny, midway, below]{$ a $};
			\draw [midarrow={>}] (B)--(C) node[font=\tiny, midway, left]{$ b $};
			\draw [dashed,midarrow={>}] (A)--(C) node[font=\tiny, midway, below]{$ a b $\,\,\,};	
			\draw [midarrow={>}] (A)--(D) node[font=\tiny, midway, left]{$ a b c $};
			\draw [midarrow={>}] (B)--(D) node[font=\tiny, midway, right]{$ b c $};
			\draw [midarrow={>}] (C)--(D) node[font=\tiny, midway, right]{$ c $};
			\draw [midarrow={>}] (A)--(E) node[font=\tiny, midway, left]{$ a b d $};
			\draw [midarrow={>}] (B)--(E) node[font=\tiny, midway, left] {$ b d $};
			\draw [midarrow={>}] (C)--(E) node[font=\tiny, midway, right]{$ d $};
			\fill[blue, opacity=.5] (A) circle (3pt);
			\fill[blue, opacity=.5] (B) circle (3pt);
			\fill[blue, opacity=.5] (C) circle (3pt);
			\draw[thick,blue,blue] (A)--(B);
			\draw[thick,blue,blue] (B)--(C);
			\draw[dashed,thick,blue,blue] (A)--(C);			
		\end{scope}
\end{tikzpicture}$

&
 
$ \left\{ 
 \begin{array}{ll}

\begin{tikzpicture}
\fill[blue,blue, opacity=.5] (0,0)--(1.9,-0.5)--(3,0.5)--(0,0);
	\node[circle, fill, inner sep=.9pt, outer sep=0pt] (A) at (0,0){};
	\node[circle, fill, inner sep=.9pt, outer sep=0pt] (B) at (1.9,-0.5){};
	\node[circle, fill, inner sep=.9pt, outer sep=0pt] (C) at (3,0.5){};
	\node[circle, fill, inner sep=.9pt, outer sep=0pt] (D) at (1.5,1.8){};
		\begin{scope}
			\draw [midarrow={>}] (A)--(B) node[font=\tiny, midway, below]{$ a $};
			\draw [midarrow={>}] (B)--(C) node[font=\tiny, midway, right]{$ b $};
			\draw [dashed,midarrow={>}] (A)--(C) node[font=\tiny, midway, below]{$ a b $\,\,\,};	
			\draw [midarrow={>}] (A)--(D) node[font=\tiny, midway, left]{$ a b c $};
			\draw [midarrow={>}] (B)--(D) node[font=\tiny, midway, right]{$ b c $};
			\draw [midarrow={>}] (C)--(D) node[font=\tiny, midway, right]{$ c $};
			\fill[blue, opacity=.5] (A) circle (3pt);
			\fill[blue, opacity=.5] (B) circle (3pt);
			\fill[blue, opacity=.5] (C) circle (3pt);
			\draw[thick,blue,blue] (A)--(B);
			\draw[thick,blue,blue] (B)--(C);
			\draw[dashed,thick,blue,blue] (A)--(C);			
		\end{scope}
\end{tikzpicture}

& \alpha_{2 , 2 , 2 , 3} ( a , b , c ) 

\\
\\

\begin{tikzpicture}
\fill[blue,blue, opacity=.5] (0,0)--(1.9,-0.5)--(3,0.5)--(0,0);
	\node[circle, fill, inner sep=.9pt, outer sep=0pt] (A) at (0,0){};
	\node[circle, fill, inner sep=.9pt, outer sep=0pt] (B) at (1.9,-0.5){};
	\node[circle, fill, inner sep=.9pt, outer sep=0pt] (C) at (3,0.5){};
	\node[circle, fill, inner sep=.9pt, outer sep=0pt] (E) at (1.5,-1.8){};
		\begin{scope}
			\draw [midarrow={>}] (A)--(B) node[font=\tiny, midway, below]{$ a $};
			\draw [midarrow={>}] (B)--(C) node[font=\tiny, midway, left]{$ b $};
			\draw [midarrow={>}] (A)--(C) node[font=\tiny, midway, above]{$ a b $\,\,\,};	
			\draw [midarrow={>}] (A)--(E) node[font=\tiny, midway, left]{$ a b d $};
			\draw [midarrow={>}] (B)--(E) node[font=\tiny, midway, left] {$ b d $};
			\draw [midarrow={>}] (C)--(E) node[font=\tiny, midway, right]{$ d $};
			\fill[blue, opacity=.5] (A) circle (3pt);
			\fill[blue, opacity=.5] (B) circle (3pt);
			\fill[blue, opacity=.5] (C) circle (3pt);
			\draw[thick,blue,blue] (A)--(B);
			\draw[thick,blue,blue] (B)--(C);
			\draw[thick,blue,blue] (A)--(C);			
		\end{scope}
\end{tikzpicture}

& \displaystyle \alpha^{-1}_{2 , 2 , 2 , 3} ( a , b , d )

\\ 
\end{array} \right.$

\\
\\

\end{tabular}
\end{figure}
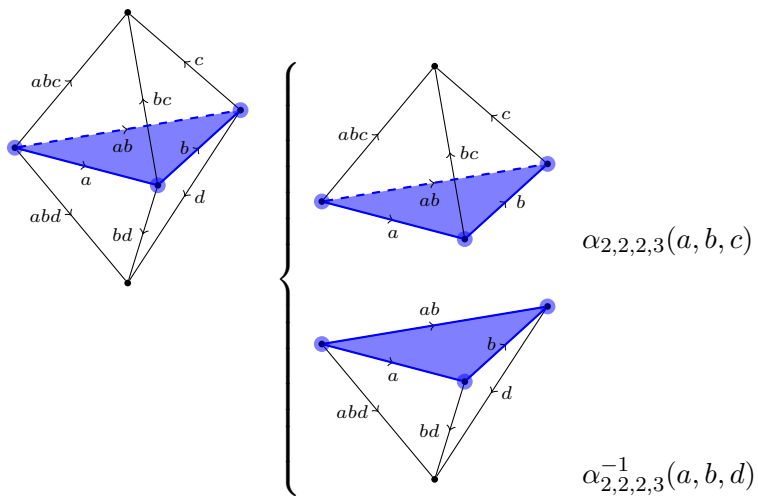

\begin{figure}[p]\caption{Six Tetrahedron State of 2-6 Move \label{six_no_knot} \vspace*{6pt}}
\begin{tabular}{ll}

$\displaystyle 
\begin{tikzpicture}
\fill[blue,blue, opacity=.5] (0,0)--(1.9,-0.5)--(3,0.5)--(0,0);
	\node[circle, fill, inner sep=.9pt, outer sep=0pt] (A) at (0,0){};
	\node[circle, fill, inner sep=.9pt, outer sep=0pt] (B) at (1.9,-0.5){};
	\node[circle, fill, inner sep=.9pt, outer sep=0pt] (C) at (3,0.5){};
	\node[circle, fill, inner sep=.9pt, outer sep=0pt] (D) at (1.5,1.8){};
	\node[circle, fill, inner sep=.9pt, outer sep=0pt] (E) at (1.5,-1.8){};
	\node[circle, fill, inner sep=.9pt, outer sep=0pt] (F) at (1.5,-0.1){};
		\begin{scope}
			\draw [midarrow={>}] (A)--(B) node[font=\tiny, midway, below]{$ a $};
			\draw [midarrow={>}] (B)--(C) node[font=\tiny, midway, below]{\!$ b $};
			\draw [dashed,midarrow={>}] (A)--(C) node[font=\tiny, midway, above]{$ a b $\,\,\,\,\,\,\,};	
			\draw [midarrow={>}] (A)--(D) node[font=\tiny, midway, left]{$ abc $};
			\draw [midarrow={>}] (B)--(D) node[font=\tiny, midway, right]{$ bc $};
			\draw [midarrow={>}] (C)--(D) node[font=\tiny, midway, right]{$ c $};
			\draw [midarrow={>}] (A)--(E) node[font=\tiny, midway, left]{$ abd $};
			\draw [midarrow={>}] (B)--(E) node[font=\tiny, midway, right] {$ bd $};
			\draw [midarrow={>}] (C)--(E) node[font=\tiny, midway, right]{$ d $};
			\draw [dashed,thick, midarrow={>}] (1.5,-0.1) -- (1.5,1.8) node[font=\tiny, midway, left]{$ g $};
			\draw [dashed,thick, midarrow={>}] (1.5,-0.25) -- (1.5,-1.8) node[font=\tiny, midway, left]{$ h $};
			\fill[blue, opacity=.5] (A) circle (3pt);
			\fill[blue, opacity=.5] (B) circle (3pt);
			\fill[blue, opacity=.5] (C) circle (3pt);
			\draw[thick,blue,blue] (A)--(B);
			\draw[thick,blue,blue] (B)--(C);
			\draw[dashed,thick,blue,blue] (A)--(C);	
			\draw [thick, midarrow={>}] (A)--(F) node[font=\tiny, midway, above]{$ abf $};
			\draw [thick, midarrow={>}] (B)--(F) node[font=\tiny, midway, below]{$ bf $\,\,};
			\draw [thick, midarrow={>}] (C)--(F) node[font=\tiny, midway, below]{\!\!\!\!\!$ f $};		
		\end{scope}
\end{tikzpicture}$

&
 
$ \left\{ 
 \begin{array}{ll}

\begin{tikzpicture}
\fill[blue,blue, opacity=.5] (0,0)--(1.9,-0.5)--(1.5,0.1)--(0,0);
	\node[circle, fill, inner sep=.9pt, outer sep=0pt] (A) at (0,0){};
	\node[circle, fill, inner sep=.9pt, outer sep=0pt] (B) at (1.9,-0.5){};
	\node[circle, fill, inner sep=.9pt, outer sep=0pt] (D) at (1.5,1.8){};
	\node[circle, fill, inner sep=.9pt, outer sep=0pt] (F) at (1.5,0.1){};
		\begin{scope}
			\draw [midarrow={>}] (A)--(B) node[font=\tiny, midway, below]{$ a $};
			\draw [midarrow={>}] (A)--(D) node[font=\tiny, midway, left]{$ abc $};
			\draw [midarrow={>}] (B)--(D) node[font=\tiny, midway, right]{$ bc $};
			\draw [dashed,thick, midarrow={>}] (1.5,0.1) -- (1.5,1.8) node[font=\tiny, midway, left]{$ g $};
			\fill[blue, opacity=.5] (A) circle (3pt);
			\fill[blue, opacity=.5] (B) circle (3pt);
			\fill[blue, opacity=.5] (F) circle (3pt);
			\draw[thick,blue,blue] (A)--(B);
			\draw [dashed,thick, midarrow={>}] (A)--(F) node[font=\tiny, midway, above]{$ abf $};
			\draw [dashed,thick, midarrow={>}] (B)--(F) node[font=\tiny, midway, right]{\,\,$ bf $};
			\draw[dashed,thick,blue,blue] (A)--(F);
			\draw[dashed,thick,blue,blue] (F)--(B);
		\end{scope}
\end{tikzpicture}

& \alpha_{2 , 2 , 2 , 3} ( a , b f , g ) 

\\
\\

\begin{tikzpicture}
\fill[blue,blue, opacity=.5] (0,0)--(1.5,-0.1)--(3,0.5)--(0,0);
	\node[circle, fill, inner sep=.9pt, outer sep=0pt] (A) at (0,0){};
	\node[circle, fill, inner sep=.9pt, outer sep=0pt] (C) at (3,0.5){};
	\node[circle, fill, inner sep=.9pt, outer sep=0pt] (D) at (1.5,1.8){};
	\node[circle, fill, inner sep=.9pt, outer sep=0pt] (F) at (1.5,-0.1){};
		\begin{scope}
			\draw [dashed,midarrow={>}] (A)--(C) node[font=\tiny, midway, above]{$ a b $\,\,\,\,\,\,\,};	
			\draw [midarrow={>}] (A)--(D) node[font=\tiny, midway, left]{$ abc $};
			\draw [midarrow={>}] (C)--(D) node[font=\tiny, midway, right]{$ c $};
			\draw [thick, midarrow={>}] (1.5,-0.1) -- (1.5,1.8) node[font=\tiny, midway, right]{$ g $};
			\fill[blue, opacity=.5] (A) circle (3pt);
			\fill[blue, opacity=.5] (C) circle (3pt);
			\fill[blue, opacity=.5] (F) circle (3pt);
			\draw[dashed,thick,blue,blue] (A)--(C);	
			\draw [thick, midarrow={>}] (A)--(F) node[font=\tiny, midway, below]{$ abf $};
			\draw [thick, midarrow={>}] (C)--(F) node[font=\tiny, midway, below]{$ f $};
			\draw[thick,blue,blue] (F)--(A);
			\draw[thick,blue,blue] (F)--(C);		
		\end{scope}
\end{tikzpicture}

& \displaystyle \alpha^{-1}_{2 , 2 , 2 , 3} ( ab , f , g )

\\ 
\\

\begin{tikzpicture}
\fill[blue,blue, opacity=.5] (1.5,-0.1)--(1.9,-0.5)--(3,0.5)--(1.5,-0.1);
	\node[circle, fill, inner sep=.9pt, outer sep=0pt] (B) at (1.9,-0.5){};
	\node[circle, fill, inner sep=.9pt, outer sep=0pt] (C) at (3,0.5){};
	\node[circle, fill, inner sep=.9pt, outer sep=0pt] (D) at (1.5,1.8){};
	\node[circle, fill, inner sep=.9pt, outer sep=0pt] (F) at (1.5,-0.1){};
		\begin{scope}
			\draw [midarrow={>}] (B)--(C) node[font=\tiny, midway, below]{\!$ b $};
			\draw [midarrow={>}] (B)--(D) node[font=\tiny, midway, right]{$ bc $};
			\draw [midarrow={>}] (C)--(D) node[font=\tiny, midway, right]{$ c $};
			\draw [thick, midarrow={>}] (1.5,-0.1) -- (1.5,1.8) node[font=\tiny, midway, left]{$ g $};
			\fill[blue, opacity=.5] (B) circle (3pt);
			\fill[blue, opacity=.5] (C) circle (3pt);
			\fill[blue, opacity=.5] (F) circle (3pt);
			\draw[thick,blue,blue] (B)--(C);
			\draw [thick, midarrow={>}] (B)--(F) node[font=\tiny, midway, below]{$ bf $\,\,\,};
			\draw [dashed,thick, midarrow={>}] (C)--(F) node[font=\tiny, midway, above]{$ f $};
			\draw[thick,blue,blue] (F)--(B);
			\draw[dashed,thick,blue,blue] (F)--(C);			
		\end{scope}
\end{tikzpicture}

&  \alpha_{2 , 2 , 2 , 3} ( b , f , g )

\\
\\

\begin{tikzpicture}
\fill[blue,blue, opacity=.5] (0,0)--(1.9,-0.5)--(1.5,0.1)--(0,0);
	\node[circle, fill, inner sep=.9pt, outer sep=0pt] (A) at (0,0){};
	\node[circle, fill, inner sep=.9pt, outer sep=0pt] (B) at (1.9,-0.5){};
	\node[circle, fill, inner sep=.9pt, outer sep=0pt] (E) at (1.5,-1.8){};
	\node[circle, fill, inner sep=.9pt, outer sep=0pt] (F) at (1.5,0.1){};
		\begin{scope}
			\draw [midarrow={>}] (A)--(B) node[font=\tiny, midway, below]{$ a $};
			\draw [midarrow={>}] (A)--(E) node[font=\tiny, midway, left]{$ abd $};
			\draw [midarrow={>}] (B)--(E) node[font=\tiny, midway, right] {$ bd $};
			\draw [dashed,thick, midarrow={>}] (F) -- (1.5,-1.8) node[font=\tiny, midway, left]{$ h $};
			\fill[blue, opacity=.5] (A) circle (3pt);
			\fill[blue, opacity=.5] (B) circle (3pt);
			\fill[blue, opacity=.5] (F) circle (3pt);
			\draw[thick,blue,blue] (A)--(B);
			\draw [thick, midarrow={>}] (A)--(F) node[font=\tiny, midway, above]{$ abf $};
			\draw [thick, midarrow={>}] (B)--(F) node[font=\tiny, midway, right]{$ bf $};
			\draw[thick,blue,blue] (F)--(B);
			\draw[thick,blue,blue] (F)--(A);	
		\end{scope}
\end{tikzpicture}

&  \alpha^{-1}_{2 , 2 , 2 , 3} ( a , b f , h )

\\
\\

\begin{tikzpicture}
\fill[blue,blue, opacity=.5] (0,0)--(1.5,-0.1)--(3,0.5)--(0,0);
	\node[circle, fill, inner sep=.9pt, outer sep=0pt] (A) at (0,0){};
	\node[circle, fill, inner sep=.9pt, outer sep=0pt] (C) at (3,0.5){};
	\node[circle, fill, inner sep=.9pt, outer sep=0pt] (E) at (1.5,-1.8){};
	\node[circle, fill, inner sep=.9pt, outer sep=0pt] (F) at (1.5,-0.1){};
		\begin{scope}
			\draw [midarrow={>}] (A)--(C) node[font=\tiny, midway, above]{$ a b $\,\,\,\,\,\,\,};	
			\draw [midarrow={>}] (A)--(E) node[font=\tiny, midway, left]{$ abd $};
			\draw [midarrow={>}] (C)--(E) node[font=\tiny, midway, right]{$ d $};
			\draw [thick, midarrow={>}] (F) -- (1.5,-1.8) node[font=\tiny, midway, left]{$ h $};
			\fill[blue, opacity=.5] (A) circle (3pt);
			\fill[blue, opacity=.5] (C) circle (3pt);
			\fill[blue, opacity=.5] (F) circle (3pt);
			\draw[dashed,thick,blue,blue] (A)--(C);	
			\draw [thick, midarrow={>}] (A)--(F) node[font=\tiny, midway, below]{$ abf $};
			\draw [thick, midarrow={>}] (C)--(F) node[font=\tiny, midway, below]{\!\!\!\!\!$ f $};	
			\draw[thick,blue,blue] (F)--(C);
			\draw[thick,blue,blue] (F)--(A);		
		\end{scope}
\end{tikzpicture}

&  \alpha_{2 , 2 , 2 , 3} ( a b , f , h )

\\
\\

\begin{tikzpicture}
\fill[blue,blue, opacity=.5] (1.5,-0.1)--(1.9,-0.5)--(3,0.5)--(1.5,-0.1);
	\node[circle, fill, inner sep=.9pt, outer sep=0pt] (B) at (1.9,-0.5){};
	\node[circle, fill, inner sep=.9pt, outer sep=0pt] (C) at (3,0.5){};
	\node[circle, fill, inner sep=.9pt, outer sep=0pt] (E) at (1.5,-1.8){};
	\node[circle, fill, inner sep=.9pt, outer sep=0pt] (F) at (1.5,-0.1){};
		\begin{scope}
			\draw [midarrow={>}] (B)--(C) node[font=\tiny, midway, below]{\!$ b $};
			\draw [midarrow={>}] (B)--(E) node[font=\tiny, midway, right] {$ bd $};
			\draw [midarrow={>}] (C)--(E) node[font=\tiny, midway, right]{$ d $};
			\draw [thick, midarrow={>}] (F) -- (1.5,-1.8) node[font=\tiny, midway, left]{$ h $};
			\fill[blue, opacity=.5] (B) circle (3pt);
			\fill[blue, opacity=.5] (C) circle (3pt);
			\fill[blue, opacity=.5] (F) circle (3pt);
			\draw[thick,blue,blue] (B)--(C);
			\draw [thick, midarrow={>}] (B)--(F) node[font=\tiny, midway, below]{$ bf $\,\,};
			\draw [thick, midarrow={>}] (C)--(F) node[font=\tiny, midway, below]{\!\!\!\!\!$ f $};	
			\draw[thick,blue,blue] (F)--(C);
			\draw[thick,blue,blue] (F)--(B);		
		\end{scope}
\end{tikzpicture}

&  \alpha^{-1}_{2 , 2 , 2 , 3} ( b , f , h )

\\

\end{array} \right.$

\\
\\

\end{tabular}
\end{figure}
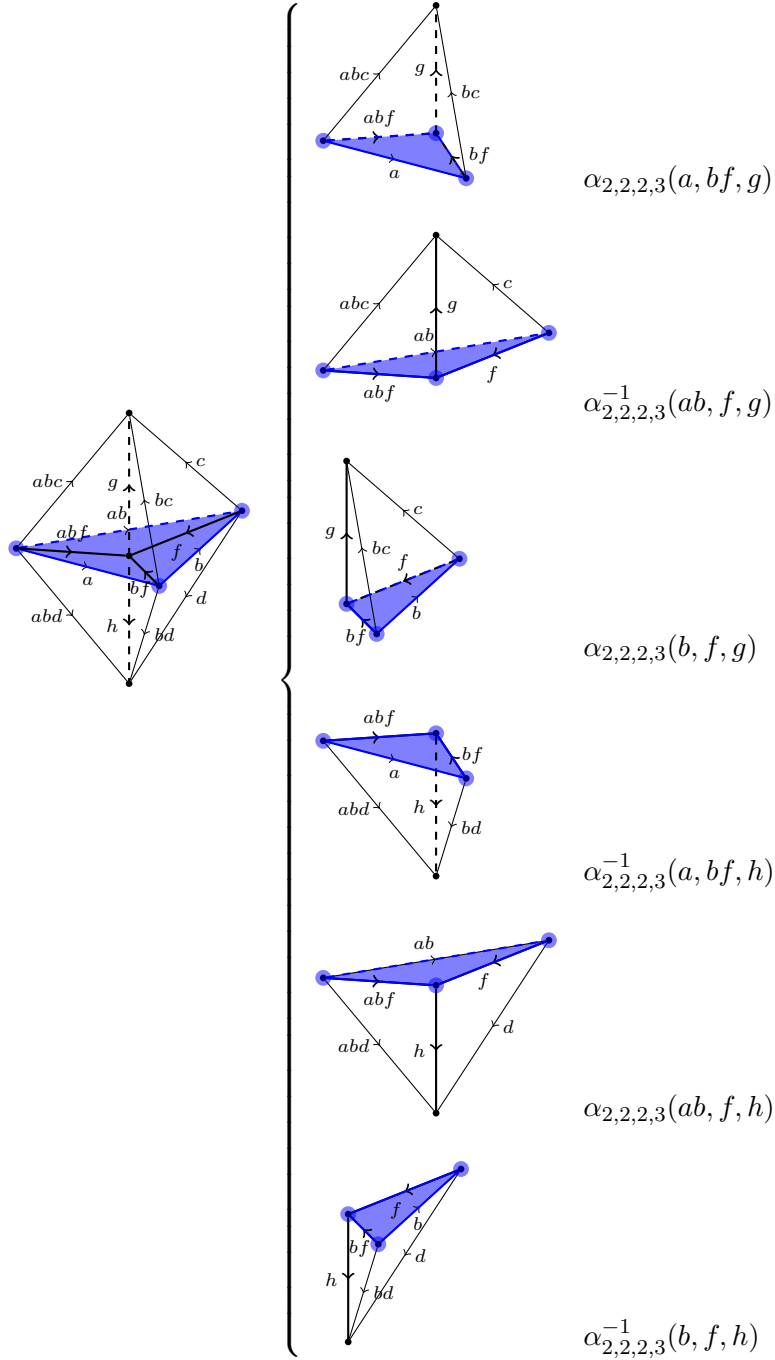

In each case, separating the factors involving $d$ (resp. $g$) from those involving $e$ (resp. $h$) gives an instance of condition (2) in the definition of partial cocycle, which is thus seen to ensure invariance under the illustrated instance of the 4-4 extended Pachner move and the 2-6 extended Pachner move.  Invariance under moves of these forms with a vertex (resp. an edge) of the polygon in the Seifert surface being modified by the move lying in the knot is, by an identical argument, given by the instances of condition (2) with $i=1$ and $j =2$ (resp. $i = j = 1$).  In the case illustrated of the 2-6 extended Pachner move, the new vertex is added to the ordering after the vertices of the triangle lying in the Seifert surface.  However, if it is inserted anywhere else in the ordering, an undefying calculation shows that an equation of the same form ensures invariance under the 2-6 move.

The case of 1-4 and 2-3 Pachner moves is similar.  Invariance under each is ensured by condition (1).  We illustrate this with the instances of the moves in which the boundary contains an edge of the knot bounding a triangle of the Seifert surface.  The before and after states of the 2-3 move are shown in Figure \ref{2-3_with_knot_and_surface}

\begin{figure}[p]\caption{An Instance of a Flag-Like 2-3 Move \label{2-3_with_knot_and_surface}}
\begin{tabular}{ll}

$\displaystyle
\begin{tikzpicture}
	\fill[blue,blue, opacity=.5] (1.5,1.8)--(0,0)--(1.9,-0.5);
        \node[circle, fill, inner sep=.9pt, outer sep=0pt] (A) at (1.5,1.8){};
	\node[circle, fill, inner sep=.9pt, outer sep=0pt] (B) at (0,0){};
	\node[circle, fill, inner sep=.9pt, outer sep=0pt] (C) at (1.9,-0.5){};
	\node[circle, fill, inner sep=.9pt, outer sep=0pt] (D) at (3,0.5){};
	\node[circle, fill, inner sep=.9pt, outer sep=0pt] (E) at (1.5,-1.8){};
		\begin{scope}
			\draw [dashed,midarrow={>}] (B)--(D) node[font=\tiny, midway, below]{$b \cdot c$};	
			\draw [midarrow={>}] (A)--(B) node[font=\tiny, midway, left]{$a$};
			\draw [midarrow={>}] (B)--(C) node[font=\tiny, midway, below]{$b$};
			\draw [midarrow={>}] (C)--(D) node[font=\tiny, midway, above]{$c$};
			\draw [midarrow={>}] (A)--(C) node[font=\tiny, midway, right]{$a \cdot b$};
			\draw [midarrow={>}] (A)--(D) node[font=\tiny, midway, right]{\,$a \cdot b \cdot c$\,};
			\draw [midarrow={>}] (B)--(E) node[font=\tiny, midway, left]{$b \cdot c \cdot d$\,};
			\draw [midarrow={>}] (C)--(E) node[font=\tiny, midway, left] {$c \cdot d$\,};
			\draw [midarrow={>}] (D)--(E) node[font=\tiny, midway, right]{$d$};
			\fill[red, opacity=.5] (A) circle (3pt);
			\fill[red, opacity=.5] (B) circle (3pt);
			\fill[blue, opacity=.5] (C) circle (3pt);
			\draw[thick,red,red, opacity=.5] (A)--(B);
			\draw[thick,blue,blue, opacity=.5] (B)--(C)--(A);	
		\end{scope}
\end{tikzpicture}$

&
 
$ \left\{ 
 \begin{array}{ll}

\begin{tikzpicture}
	\fill[blue,blue, opacity=.5] (1.5,1.8)--(0,0)--(1.9,-0.5);
        \node[circle, fill, inner sep=.9pt, outer sep=0pt] (A) at (1.5,1.8){};
	\node[circle, fill, inner sep=.9pt, outer sep=0pt] (B) at (0,0){};
	\node[circle, fill, inner sep=.9pt, outer sep=0pt] (C) at (1.9,-0.5){};
	\node[circle, fill, inner sep=.9pt, outer sep=0pt] (D) at (3,0.5){};
		\begin{scope}
			\draw [dashed,midarrow={>}] (B)--(D) node[font=\tiny, midway, below]{$b \cdot c$};	
			\draw [midarrow={>}] (A)--(B) node[font=\tiny, midway, left]{$a$};
			\draw [midarrow={>}] (B)--(C) node[font=\tiny, midway, below]{$b$};
			\draw [midarrow={>}] (C)--(D) node[font=\tiny, midway, above]{$c$};
			\draw [midarrow={>}] (A)--(C) node[font=\tiny, midway, right]{$a \cdot b$};
			\draw [midarrow={>}] (A)--(D) node[font=\tiny, midway, right]{\,$a \cdot b \cdot c$\,};
			\fill[red, opacity=.5] (A) circle (3pt);
			\fill[red, opacity=.5] (B) circle (3pt);
			\fill[blue, opacity=.5] (C) circle (3pt);
			\draw[thick,red,red, opacity=.5] (A)--(B);
			\draw[thick,blue,blue, opacity=.5] (B)--(C)--(A);		
		\end{scope}
\end{tikzpicture}

& \alpha^{-1}_{1 , 1 , 2 , 3} ( a , b , c ) 
\\
\\

\begin{tikzpicture}
	\node[circle, fill, inner sep=.9pt, outer sep=0pt] (B) at (0,0){};
	\node[circle, fill, inner sep=.9pt, outer sep=0pt] (C) at (1.9,-0.5){};
	\node[circle, fill, inner sep=.9pt, outer sep=0pt] (D) at (3,0.5){};
	\node[circle, fill, inner sep=.9pt, outer sep=0pt] (E) at (1.5,-1.8){};
		\begin{scope}
			\draw [midarrow={>}] (B)--(D) node[font=\tiny, midway, below]{$b \cdot c$};	
			\draw [midarrow={>}] (B)--(C) node[font=\tiny, midway, below]{$b$};
			\draw [midarrow={>}] (C)--(D) node[font=\tiny, midway, above]{$c$};
			\draw [midarrow={>}] (B)--(E) node[font=\tiny, midway, left]{$b \cdot c \cdot d$\,};
			\draw [midarrow={>}] (C)--(E) node[font=\tiny, midway, left] {$c \cdot d$\,};
			\draw [midarrow={>}] (D)--(E) node[font=\tiny, midway, right]{$d$};
			\fill[red, opacity=.5] (B) circle (3pt);
			\fill[blue, opacity=.5] (C) circle (3pt);
			\draw[thick,blue,blue, opacity=.5] (B)--(C);
		\end{scope}
\end{tikzpicture}

& \displaystyle \alpha^{-1}_{1 , 2 , 3 , 3} ( b , c , d )
\\
\end{array} \right.$

\\
\\

$\displaystyle 
\begin{tikzpicture}
	\fill[blue,blue, opacity=.5] (1.5,1.8)--(0,0)--(1.9,-0.5);
        \node[circle, fill, inner sep=.9pt, outer sep=0pt] (A) at (1.5,1.8){};
	\node[circle, fill, inner sep=.9pt, outer sep=0pt] (B) at (0,0){};
	\node[circle, fill, inner sep=.9pt, outer sep=0pt] (C) at (1.9,-0.5){};
	\node[circle, fill, inner sep=.9pt, outer sep=0pt] (D) at (3,0.5){};
	\node[circle, fill, inner sep=.9pt, outer sep=0pt] (E) at (1.5,-1.8){};
		\begin{scope}
			\draw [dashed,midarrow={>}] (B)--(D) node[font=\tiny, midway, left]{$bc$};	
			\draw [midarrow={>}] (A)--(B) node[font=\tiny, midway, left]{$a$};
			\draw [midarrow={>}] (B)--(C) node[font=\tiny, midway, below]{$b$};
			\draw [midarrow={>}] (C)--(D) node[font=\tiny, midway, above]{$c$};
			\draw [midarrow={>}] (A)--(C) node[font=\tiny, midway, right]{$ab$};
			\draw [midarrow={>}] (A)--(D) node[font=\tiny, midway, right]{\,$abc$\,};
			\draw [midarrow={>}] (B)--(E) node[font=\tiny, midway, left]{$bcd$\,};
			\draw [midarrow={>}] (C)--(E) node[font=\tiny, midway, right] {$cd$\,};
			\draw [midarrow={>}] (D)--(E) node[font=\tiny, midway, right]{$d$};
			\draw [dashed,thick, midarrow={>}] (1.5,1.8) -- (1.5,-0.2) node[font=\tiny, midway, left]{$abcd$};
			\draw [dashed,thick] (1.5,-0.6) -- (1.5,-1.8);		
			\fill[red, opacity=.5] (A) circle (3pt);
			\fill[red, opacity=.5] (B) circle (3pt);
			\fill[blue, opacity=.5] (C) circle (3pt);
			\draw[thick,red,red, opacity=.5] (A)--(B);
			\draw[thick,blue,blue, opacity=.5] (B)--(C)--(A);	
		\end{scope}
\end{tikzpicture}$

&
 
$ \left\{ 
 \begin{array}{ll}

\begin{tikzpicture}
	\fill[blue,blue, opacity=.5] (1.5,1.8)--(0,0)--(1.9,-0.5);
        \node[circle, fill, inner sep=.9pt, outer sep=0pt] (A) at (1.5,1.8){};
	\node[circle, fill, inner sep=.9pt, outer sep=0pt] (B) at (0,0){};
	\node[circle, fill, inner sep=.9pt, outer sep=0pt] (C) at (1.9,-0.5){};
	\node[circle, fill, inner sep=.9pt, outer sep=0pt] (E) at (1.5,-1.8){};
		\begin{scope}
			\draw [midarrow={>}] (A)--(B) node[font=\tiny, midway, left]{$a$};
			\draw [midarrow={>}] (B)--(C) node[font=\tiny, midway, below]{$b$};
			\draw [midarrow={>}] (A)--(C) node[font=\tiny, midway, right]{$ab$};
			\draw [midarrow={>}] (B)--(E) node[font=\tiny, midway, left]{$bcd$\,};
			\draw [midarrow={>}] (C)--(E) node[font=\tiny, midway, right] {$cd$\,};
			\draw [dashed,thick, midarrow={>}] (1.5,1.8) -- (1.5,-0.2) node[font=\tiny, midway, left]{$abcd$};
			\draw [dashed,thick] (1.5,-0.6) -- (1.5,-1.8);
			\fill[red, opacity=.5] (A) circle (3pt);
			\fill[red, opacity=.5] (B) circle (3pt);
			\fill[blue, opacity=.5] (C) circle (3pt);
			\draw[thick,red,red, opacity=.5] (A)--(B);
			\draw[thick,blue,blue, opacity=.5] (B)--(C)--(A);		
		\end{scope}
\end{tikzpicture}

& \alpha^{-1}_{1 , 1 , 2 , 3} ( a , b , c d ) 

\\
\\

\begin{tikzpicture}
	\node[circle, fill, inner sep=.9pt, outer sep=0pt] (A) at (1.5,1.8){};
	\node[circle, fill, inner sep=.9pt, outer sep=0pt] (C) at (1.9,-0.5){};
	\node[circle, fill, inner sep=.9pt, outer sep=0pt] (D) at (3,0.5){};
	\node[circle, fill, inner sep=.9pt, outer sep=0pt] (E) at (1.5,-1.8){};
		\begin{scope}
			\draw [midarrow={>}] (C)--(D) node[font=\tiny, midway, above]{$c$};
			\draw [midarrow={>}] (A)--(C) node[font=\tiny, midway, right]{$ab$};
			\draw [midarrow={>}] (A)--(D) node[font=\tiny, midway, right]{\,$abc$\,};
			\draw [midarrow={>}] (C)--(E) node[font=\tiny, midway, right] {$cd$\,};
			\draw [midarrow={>}] (D)--(E) node[font=\tiny, midway, right]{$d$};
			\draw [midarrow={>}] (A)--(E) node[font=\tiny, midway, left]{$abcd$};
			\fill[red, opacity=.5] (A) circle (3pt);
			\fill[blue, opacity=.5] (C) circle (3pt);
			\draw[thick,blue,blue, opacity=.5] (A)--(C);
		\end{scope}
\end{tikzpicture}

& \displaystyle \alpha^{-1}_{1 , 2 , 3 , 3} ( a b , c , d )

\\ 
\\

\begin{tikzpicture}
         \node[circle, fill, inner sep=.9pt, outer sep=0pt] (A) at (1.5,1.8){};
	\node[circle, fill, inner sep=.9pt, outer sep=0pt] (B) at (0,0){};
	\node[circle, fill, inner sep=.9pt, outer sep=0pt] (D) at (3,0.5){};
	\node[circle, fill, inner sep=.9pt, outer sep=0pt] (E) at (1.5,-1.8){};
		\begin{scope}
			\draw [dashed,midarrow={>}] (B)--(D) node[font=\tiny, midway, left]{$bc$};	
			\draw [midarrow={>}] (A)--(B) node[font=\tiny, midway, left]{$a$};
			\draw [midarrow={>}] (A)--(D) node[font=\tiny, midway, right]{\,$abc$\,};
			\draw [midarrow={>}] (B)--(E) node[font=\tiny, midway, left]{$bcd$\,};
			\draw [midarrow={>}] (D)--(E) node[font=\tiny, midway, right]{$d$};
			\draw [midarrow={>}] (A)--(E) node[font=\tiny, midway, right]{$abcd$};
			\fill[red, opacity=.5] (A) circle (3pt);
			\fill[red, opacity=.5] (B) circle (3pt);
			\draw[thick,red,red, opacity=.5] (A)--(B);	
		\end{scope}
\end{tikzpicture}

&  \alpha_{1 , 1 , 3 , 3} ( a , b c, d )

\\
\end{array} \right.$

\end{tabular}
\end{figure}
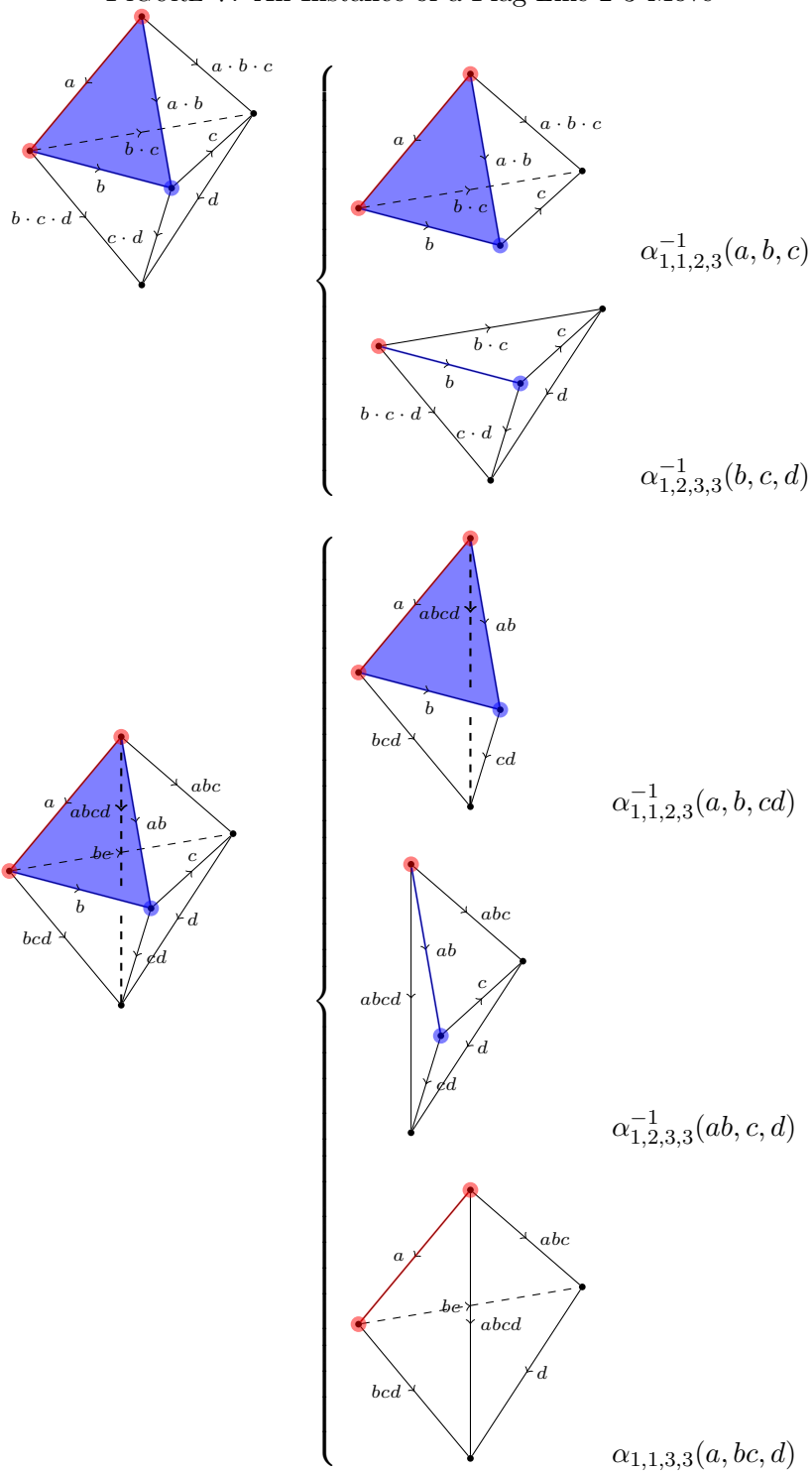

Plainly invariance under this move is equivalent to the relation
\begin{eqnarray*}
\lefteqn{ \alpha^{-1}_{1 , 1 , 2 , 3} ( a , b , c ) \, \cdot \, \alpha^{-1}_{1 , 2 , 3 , 3} ( b , c , d )  =} \\
& &  \alpha^{-1}_{1 , 1 , 2 , 3} ( a , b , c d ) \, \cdot \, \alpha^{-1}_{1 , 2 , 3 , 3} ( a b , c , d ) \, \cdot \, \alpha_{1 , 1 , 3 , 3} ( a , b c, d )   \end{eqnarray*}

\noindent which solves to give the instance of condition (1) with indices 1,1,2,3,3.  A little effort shows that the same relation suffices to give invariance under the other instance of a 2-3 move with an triangle on the boundary lying in the Seifert surface, with an edge on the knot -- in which the edge of the knot lies on the equatorial triangle, and under a 1-4 move in which the tetrahedron subdivided similarly has a triangle of its boundary lying in the Seifert surface with an edge on the knot, regardless of whether the new vertex is inserted into the ordering before or after the original vertex lying in the 3-dimensional stratum.

Likewise the other instances of condition (1) with indices 3,3,3,3,3 (resp. 2,3,3,3,3; 1,3,3,3,3; 2,2,3,3,3; 1,2,3,3,3; 1,1,3,3,3; 2,2,2,3,3; 1,2,2,3,3)  give invariance under 2-3 and 1-4 moves in which the boundary of the cell modified is not incident with either the knot or Seifert sufrace (resp. intersect the Seifert surface in a single vertex not on the knot; intersect the knot in a single vertex; is not incident with the knot, but intersects the Seifert surface in an edge; intersects the knot in a single vertex and the Seifert surface in an edge ending at that vertex; intersects the knot in a single edge; is not incident with the knot, but intersects the Seifert surface in a triangle; intersect the Seifert surface in a triangle with a single vertex on the knot).
Again, inserting the new vertex in 1-4 moves anywhere in the ordering gives rise to different instances of the same relation.

Thus the value of the expression is independent of both the triangulation and the ordering(s) of the vertices.
\end{proof}

As with the untwisted invariant, which is the special case of the preceding in which $\alpha \equiv 1$, by general principles explained in \cite{Y}, this invariant is the restriction to the endomorphisms of the monoidal identity, the empty triple of spaces, of a generalized TQFT, a monoidal functor from (3,2,1)-{\bf COBORD}, the category whose objects are triples of an oriented surface containing an oriented curve with boundary, with disjoint union as the monoidal product, and whose arrows are PL-homeomorphism classes of cobordisms between these, to $k$-{\bf v.s.}.

\section{Examples of Initial Data}

Finite $\mathsf 3$-parcels are easy to construct.  Given any three finite groups $G_i$ $i = 1, 2,3$ and finite sets $X_{1,2}$ and $X_{2,3}$, with $X_{i,j}$ equipped with a left $G_i$-action and a right $G_j$ action, which commute with each other, there is a $\mathsf 3$-parcel $\mathcal C$ with ${\mathcal C}(i,i) = G_i$, ${\mathcal C}(i,j) = X_{i,j}$ for $(i,j) = (1,2)$ or $(2,3)$ and 

\[ {\mathcal C}(1,3) = X_{1,2} \times X_{2,3} / \sim \]

\noindent where $\sim$ is the equivalence relation induced by

\[ (xg,y) \sim (x,gy) \]

\noindent whenever $x\in X_{1,2}$, $g\in G_2$ and $y\in X_{2,3}$, where the actions of the group are denoted by the null infix.  Composition is given by the group multiplications, group actions, and the map
$(x,y) \mapsto [x,y]$ where $[x,y]$ is the equivalence class of the pair under $\sim$ for $x\in X_{1,2}$ and $y\in X_{2,3}$, as appropriate.  The identity arrows for the object $i$ is the group identity of $G_i$.  It is trivial to verify associativity.

If we add to the data a finite set $X_{1,3}$ with a commuting left $G_1$ action and a right $G_3$ action, together with an biequivariant map  $\varphi: X_{1,2} \times X_{2,3} / \sim \rightarrow X_{1,3}$, and let the composition arrows $x:1\rightarrow 2$ and $y:2\rightarrow 3$ be given by $\varphi([x,y])$, we obtain a complete description of all finite $\mathsf 3$-parcels.

Among these are $\mathsf 3$-parcels constructed from subgroups and subsets of a groups with suitable closure properties after the manner of \cite{DPY} Example 5.2, on which it is, moreover, easy to find examples of partial 3-cocycles using ordinary group cohomology:

\begin{example}
Fix a group $G$.  Let $G_i$ $i = 1,2,3$ be subgroups of $G$ and $X_{1,2}$ (resp. $X_{2,3}$) be a subset of $G$ closed under left multiplication by elements of $G_1$ (resp. $G_2$) and right multiplication by elements of $G_2$ (resp. $G_3$), and let $X_{1,3} := \{ x\xi \; |\;  x \in X_{1,2}\;  \&\;  \xi \in X_{2,3} \}$.

The disjoint union of the $G_i$ for $i = 1,2,3$ and the $X_{i,j}$ for $(i,j) = (1,2), (2,3), (1,3)$ is easily seen to be a $\mathsf 3$-parcel $\mathcal C$ under a composition induced by the group law on $G$ with the obvious underlying functor to $\mathsf 3$.  It is also easy to see that given any $K^\times$-valued 3-cocycle $\alpha$ on $G$, the map induced on $Arr({\mathcal C})$ by composing the canonical maps to $G^3$, induced by the inclusions of the homsets into $G$, with $\alpha$ is a 3-cocycle on $\mathcal C$, and thus a partial 3-cocycle.
\end{example}


\begin{thebibliography}{1}

\bibitem{A} Alexander,  J., ``The Combinatorial Theory of Complexes,'' {\em Ann. of Math.} {\bf 31} (2) (1930) 292-320.

\bibitem{C} Casali, Maria Rita, ``A note about bistellar operations on PL-manifolds with boundary,''  {\em Geometria Dedicata} {\bf 56}  (1995) 257-262.

\bibitem{CY} Crane, L.  and Yetter, D.N., ``Moves on Filtered PL Manifolds and Stratified PL Spaces,'' \#arXiv:1404.3142.

\bibitem{DPY} Dougherty, Aria L., Park, Hwajin, and Yetter, D.N.,  ``On 2-dimensional Dikjgraaf-Witten Theory with Defects,'' (2014)

\bibitem{Gr1} Grandis, Marco,  ``Directed Homotopy Theory, I. The fundamental category'', {\em Cah. Topol. G\'{e}om. Diff. Cat\'{e}g.}  {\bf 44} (2003), 281-316. 

\bibitem{Gr2} Grandis, Marco, {\em Directed Algebraic Topology: Models of non-reversible worlds}, Cambridge University Press, (2009). 

\bibitem{W} Wakui, Michihisa, ``On Dijkgraaf-Witten Invariant for 3-manifolds,'' {\em Osaka J. Math.} {\bf 29} (4) (1992) 675-696.

\bibitem{Y} Yetter, D.N., ``Triangulations and TQFTs,'' in {\em Conference Proceedings on Quantum Topology}, (Randy Baadhio, ed.), World Scientific (1993) 354-370.

\end{thebibliography}
\end{document}